\documentclass[smallcondensed]{svjour3}
\smartqed

\usepackage{latexsym}
\usepackage{amsmath}
\usepackage{amssymb}
\usepackage{amsfonts}
\usepackage{verbatim}
\usepackage{graphicx}
\usepackage{epstopdf}
\usepackage{subfigure}
\usepackage{float}
\usepackage{multirow}
\usepackage{color}
\usepackage[misc]{ifsym}

\usepackage[backref,colorlinks,linkcolor=red,anchorcolor=green,citecolor=blue]{hyperref}
\newtheorem{assump}{Assumption}

\renewcommand{\theequation}{\arabic{section}.\arabic{equation}}
\let \ssection=\section
\renewcommand{\section}{\setcounter{equation}{0}\ssection}
\DeclareMathOperator*{\esssup}{ess\,sup}

\newcommand{\<}{\langle}
\renewcommand{\>}{\rangle}
\newcommand{\nn}{\nonumber}

\journalname{Journal of Scientific Computing}
\date{\today}

\begin{document}


  \title{On Efficient Second Order Stabilized Semi-Implicit
    Schemes for the Cahn-Hilliard Phase-Field Equation}
   \titlerunning{Efficient Second Order Stabilized Semi-Implicit
   	Schemes for Cahn-Hilliard Equation}


%

\author{Lin Wang \and Haijun Yu}
 \institute{L. Wang\and H. Yu (\Letter)\\
 	\email{wanglin@lsec.cc.ac.cn} (L. Wang),
 	\email{hyu@lsec.cc.ac.cn} (H. Yu)\\
     School of Mathematical Sciences,
     University of Chinese Academy of Sciences, Beijing
     100049, China.\\
 	NCMIS \& LSEC, Institute of
 	Computational Mathematics and Scientific/Engineering
 	Computing, Academy of Mathematics and Systems
 	Science, Beijing 100190, China.\\
 	}

\maketitle

\begin{abstract}
  Efficient and energy stable high order time marching
  schemes are very important but not easy to construct for
  the study of nonlinear phase dynamics. In this paper, we
  propose and study two linearly stabilized second order semi-implicit
  schemes for the Cahn-Hilliard phase-field equation. 
  One uses backward differentiation formula and the other
  uses Crank-Nicolson method to discretize linear
  terms. In both schemes, the nonlinear bulk forces are
  treated explicitly with two second-order stabilization terms. This treatment leads to linear elliptic systems with
  constant coefficients, for which lots of robust and efficient solvers are available. The discrete energy dissipation properties are  proved for both schemes. Rigorous error analysis is carried out to show that, when the time step-size is small enough, second
  order accuracy in time is obtained with a prefactor
  controlled by a fixed power of $1/\varepsilon$, where
  $\varepsilon$ is the characteristic interface
  thickness. Numerical results are presented to verify the
  accuracy and efficiency of proposed schemes.
\end{abstract}
%



\section{Introduction}

In this paper, we consider numerical approximation for the
Cahn-Hilliard equation
\begin{equation}\label{eq:CH}
  \begin{cases}
    \phi_{t}=\gamma\Delta \mu,   & (x,t)\in \Omega \times (0,T],\\
    \mu=-\varepsilon \Delta \phi + \dfrac{1}{\varepsilon}f(\phi),  &(x,t)\in \Omega \times (0,T],\\
    \phi |_{t=0} =\phi_{0}(x), &x \in \Omega,
  \end{cases}
\end{equation}
with Neumann boundary condition
\begin{equation}\label{eq:CH:nbc}
	\partial_{n}\phi =0, \quad \partial_{n}\mu =0, \quad
x\in \partial\Omega.
\end{equation} 
Here $\Omega \in R^{d}, d=2,3$ is a bounded domain with a
locally Lipschitz boundary, $n$ is the outward normal, $T$
is a given time, $\phi(x,t)$ is the phase-field variable.
Function $f(\phi)=F'(\phi)$, and $F(\phi)$ is a given
energy potential with two local minima, e.g. the double well
potential $F(\phi)=\frac{1}{4}(\phi^2-1)^2$. The two minima
of $F$ produce two phases, with the typical thickness of the
interface between two phases given by
$\varepsilon$. $\gamma$, called mobility, is related to the
characteristic relaxation time of the system.

The Cahn-Hilliard equation originated from the work by Cahn
and Hilliard ~\cite{cahn_free_1958}, in which a diffusive model of interfacial energy is built to describe the phase
separation and coarsening phenomena in non-uniform systems. If the term
$\Delta \mu$ in equation \eqref{eq:CH} is replaced with
$-\mu$, one get the Allen-Cahn equation, which was
introduced by Allen and Cahn~\cite{allen_microscopic_1979}
to describe the motion of anti-phase boundaries in
crystalline solids.  The Cahn-Hilliard equation and the
Allen-Cahn equation are two widely used phase-field
models. In a phase-field model, the information of interface
is implicitly determined by a smooth phase function $\phi$. In most parts
of the domain $\Omega$, the value of $\phi$ is close to
local minima of $F$. The interface is a thin layer of
thickness $\varepsilon$ connecting regions of different
local minima.  It is easier to deal with dynamical process
involving morphology changes of interfaces using phase-field
models due to the good mathematical properties that
phase-field equations have. For this reason, phase field
models have been the subject of many theoretical and
numerical investigations for several decades(cf., for
instance, \cite{du_numerical_1991},
\cite{elliott_error_1992}, \cite{chen_spectrum_1994},
\cite{caffarelli_l_1995}, \cite{elliott_cahnhilliard_1996},
\cite{eyre_unconditionally_1998},
\cite{furihata_stable_2001}, \cite{liu_phase_2003},
\cite{feng_error_2004}, \cite{kessler_posteriori_2004},
\cite{shen_numerical_2010}, \cite{condette_spectral_2011}).

But, numerically solving the phase-field equations is not an
easy task, especially for the Cahn-Hilliard equation. Firstly, the small parameter $\varepsilon$
requires very high spatial and temporal grid
resolutions. Secondly, the small parameter $\varepsilon$ in
the nonlinear bulk force $f(\phi)$ and the bi-harmonic
operator makes the equation very stiff, which make it very hard to solve.  
Nevertheless,
lots of numerical schemes have been proposed to solve the
Cahn-Hilliard equation based on its mathematical
properties.  Two most important properties of
the Cahn-Hilliard equation are the volume conservation property 
\begin{equation}\label{eq:CH:vol}
\int_{\Omega} \phi(x,t)dx = \int_{\Omega} \phi_0(x)dx,
\quad \forall\,t>0,
\end{equation}
and energy dissipation property
\begin{equation}\label{eq:CH:Edis}
E_{\varepsilon}(\phi(\cdot,t)) - E_{\varepsilon}(\phi_0)
= - \gamma\int_0^t\!\!\int_{\Omega}
|\nabla\mu|^{2}dx
= - \gamma\int_0^t
\|\phi_t\|_{-1}^{2}dx,\quad \forall\,t>0,
\end{equation}
where $E_\varepsilon$ is the free energy functional defined as
\begin{equation}\label{eq:CH:E}
  E_{\varepsilon}(\phi):=\int_{\Omega}\Big(\frac{\varepsilon}{2}|\nabla \phi|^{2} +\frac{1}{\varepsilon} F(\phi)\Big)dx.
\end{equation}
Here the $H^{-1}$ norm $\|\cdot\|_{-1}$ is defined in Section 2. 
Since the nonlinear bulk energy $F$ is neither a convex nor a
concave function, treating it fully explicit or implicit in
a time discretization will not lead to an efficient
scheme. In fact, if the nonlinear force $f$ is treated fully
explicitly, the resulting scheme will require a very tiny
step-size to be stable (cf. for instance
\cite{shen_numerical_2010}). On the other hand, treating it
fully implicitly will lead to a nonlinear system, for which
the solution existence and uniqueness requires a restriction
on step-size as well (cf. e.g. \cite{feng_error_2004}). One
popular approach to solve this dilemma is the convex
splitting method, which appears to be introduced by Elliott and Stuart~\cite{elliott_global_1993} and popularized by
Eyre~\cite{eyre_unconditionally_1998}. In a convex splitting approach, the convex
part of $F$ is treated implicitly and the concave part
treated explicitly. The convex splitting scheme given in \cite{elliott_global_1993}\cite{eyre_unconditionally_1998} is of first order accurate
and unconditional stable. In each time step, one need solve
a nonlinear system. The solution existence and uniqueness is
guaranteed since the nonlinear system corresponds to a
convex optimization problem. The convex splitting method was
used widely, and several second order extensions were
derived in different situations
\cite{condette_spectral_2011,baskaran_energy_2013,chen_linear_2014,guo_h2_2016}.  Another type of
energy stable schemes is the secant-line method
proposed by Du and Nicolaides~\cite{du_numerical_1991}. It is also used and
extended in several other works,
e.g. \cite{furihata_stable_2001,kim_conservative_2004,feng_fully_2006,condette_spectral_2011,gomez_provably_2011,baskaran_energy_2013,zhang_adaptive_2013,benesova_implicit_2014}. Like the fully implicit method,
the usual second order convex splitting method and the
secant-type method for Cahn-Hilliard equation need a small
time step-size to guarantee the semi-discretized nonlinear
system has a unique solution (cf. for instance
\cite{du_numerical_1991,barrett_finite_1999}). To remove
the restriction on time step-size, diffusive three-step
Crank-Nicolson schemes coupled with a second order
convex splitting were introduced in \cite{guo_h2_2016}
and \cite{diegel_stability_2016}, In which, the time semi-discretized system is a nonlinear but unique solvable problem.

Recently, a new approach termed as invariant energy
quadratization (IEQ) was introduced to handle the nonlinear
energy. When applying to Cahn-Hilliard equation, it first
appeared in
\cite{guillen-gonzalez_linear_2013,guillen-gonzalez_second_2014}
as a Lagrange multiplier method. It then generalized by Yang
et al. and successfully extended to handle several very
complicated nonlinear phase-field models
\cite{yang_linear_2016,han_numerical_2017,yang_efficient_2017,yang_yu_efficient_2017}. In
the IEQ approach, a new variable which equals to the square
root of $F$ is introduced, so the energy is written into a
quadratic form in terms of the new variable. By using
semi-implicit treatments to all the nonlinear terms in the
equations, one get a linear and energy stable scheme. It is
straightforward to prove the unconditional stability for
both first order and second order IEQ schemes. Comparing to
the convex splitting approach, IEQ leads to well-structured
linear system which is easier to solve. The modified energy
in IEQ is an order-consistent approximation to the original
system energy. At each time step, it needs to solve a linear
system with time-varying coefficients. To avoid the variable-coefficient system, a new approach called scalar auxiliary variable (SAV) was introduced by Shen et al.\cite{shen_scalar_2017,shen_new_2017} recently. The methodology
of SAV is very similar to IEQ, but at each time step, only systems with constant coefficients need to be solved.

Another trend of improving numerical schemes for phase-field
models focuses on algorithm efficiency. Chen and Shen \cite{chen_applications_1998} and Zhu et al.
\cite{zhu_coarsening_1999} studied
stabilized semi-implicit Fourier-spectral method for
Cahn-Hilliard equation. The space variables are discretized
using a Fourier-spectral method whose convergence rate is
exponential in contrast to the second order of a
usual finite-difference method. The time variable is
discretized by using semi-implicit schemes which allow much
larger time step sizes than explicit schemes. Xu and Tang
\cite{xu_stability_2006} introduced a different stabilized
term to build large time-stepping stabilized semi-implicit
method for a 2-dimensional epitaxial growth model. He et al
\cite{he_large_2007} proposed a similar large time-stepping
methods for the Cahn-Hilliard equation, in which a
stabilized term $B(\phi^{n+1}-\phi^n)$
(resp. $B(\phi^{n+1}-2\phi^n+\phi^{n-1})$) is added to the
nonlinear bulk force for the first order(resp. second order)
scheme. Shen and Yang applied similar stabilization skill to
Allen-Cahn equation and Cahn-Hilliard equation in mixed
formulation \cite{shen_numerical_2010}, which leads to
unconditionally energy stable first-order linear schemes and
second-order linear schemes with reasonable stability
conditions. This idea was followed up in
\cite{feng_stabilized_2013} for the stabilized
Crank-Nicolson schemes for phase field models.  Another
stabilized second-order Crank-Nicolson scheme with a new
convex-concave splitting of the energy is proposed for a
tumor-growth system by Wu et al.\cite{wu_stabilized_2014}.
Those time marching schemes all lead to linear systems,
which are easier to solve than nonlinear systems resulting
from traditional convex-splitting schemes, in which the
nonlinear convex force is treated implicitly.  On the other
hand, when the nonlinear force is treated explicitly, one
need to introduce a proper stabilization term and a suitably
truncated nonlinear function $\tilde{f}(\phi)$ instead of
$f(\phi)$ to prove the unconditionally energy stable
property with a reasonable stabilization constant. It is
worth to mention that with no truncation made to $f(\phi)$,
Li et al \cite{li_characterizing_2016,li_second_2017}
proved that the energy stable property can be obtained as
well, but a much larger stability constant need be used.
The stabilization skill has also been used in constructing 
higher order schemes, for example the exponential time differencing (ETD) scheme~\cite{ju_fast_205} and Runge-Kutta scheme~\cite{guo_efficient_2016,shin_unconditionally_2017}.

In this paper, we study the stability and convergence properties of two new second-order semi-implicit time marching
schemes.  One uses second-order backward
differentiation formula (BDF2)  and other one uses Crank-Nicolson approximation. In both schemes, explicit extrapolation are
used for the nonlinear force with two order-consistent extra stabilization
terms added to
guarantee the energy dissipation. We also give an optimal error analysis
in $l^{\infty}(0,T;H^{-1})\cap l^{2}(0,T;H^{1})$ norm.  The
new methods have several merits: 1) They are second order
accurate; 2) They lead to linear systems with constant
coefficients after time discretization; 3) Discrete energy dissipations are proved. The proofs base on Galerkin formulation. Both finite element method and spectral method can be used for spatial
discretization to conserve volume fraction and satisfy
discretized energy dissipation law.

The remain part of this paper is organized as follows. In
Section 2, we present the two second-order stabilized
schemes for the Cahn-Hilliard equation and prove they are
energy stable. In Section 3, we present an error estimate for the BDF2 scheme to derive a convergence rate that does not depend on
$1/\varepsilon$ exponentially. Implementation details and
numerical results for a test problem in a 2-dimensional square
domain are presented in Section 4 to verify our theoretical
results. We end the paper with some concluding remarks in
Section~5.

\section{The two second order stabilized linear schemes}

We first introduce some notations which will be used
throughout the paper. We use $\|\cdot\|_{m,p}$ to denote the
standard norm of the Sobolev space $W^{m,p}(\Omega)$. In
particular, we use $\|\cdot\|_{L^p}$ to denote the norm of
$W^{0,p}(\Omega)=L^{p}(\Omega)$; $\|\cdot\|_{{m}}$ to denote
the norm of $W^{m,2}(\Omega)=H^{m}(\Omega)$; and $\|\cdot\|$
to denote the norm of $W^{0,2}(\Omega)=L^{2}(\Omega)$.  Let
$(\cdot, \cdot)$ represent the $L^{2}$ inner product.  In
addition, define for $p\geq 0$
\[
H^{-p}(\Omega):=\left(H^{p}(\Omega)\right)^{*},\quad
H_{0}^{-p}(\Omega):=\left\{ u \in H^{-p}(\Omega)
  \mid\,\<u,1\>_{p}=0 \right\},
\]
where $\<\cdot,\cdot\>_{p}$ stands for the dual product
between $H^{p}(\Omega)$ and $H^{-p}(\Omega)$. We denote
$L_{0}^{2}(\Omega):= H_{0}^{0}(\Omega)$. For
$v \in L_{0}^{2}(\Omega)$, let
$-\Delta^{-1}v:=v_{1} \in H^{1}(\Omega)\cap
L_{0}^{2}(\Omega)$, where $v_{1}$ is the solution to
\begin{equation}\nn
  -\Delta v_{1}=v \ \ {\rm in}\  \Omega ,\quad \ \frac{\partial v_{1}}{\partial n}=0 \ \ {\rm on}\  \partial \Omega,
\end{equation}
and $\|v\|_{-1}:=\sqrt{(v,-\Delta^{-1}v) }$.

For any given function $\phi(t)$ of $t$, we use $\phi^n$ to
denote an approximation of $\phi(n\tau)$, where $\tau$ is
the step-size. We will frequently use the shorthand
notations: $\delta_{t}\phi^{n+1}:=\phi^{n+1}-\phi^{n}$,
$\delta_{tt}\phi^{n+1}:=\phi^{n+1}-2\phi^{n}+\phi^{n-1}$,
$D_{\tau}\phi^{n+1}:=
\frac{3\phi^{n+1}-4\phi^{n}+\phi^{n-1}}{2\tau}
=\frac{1}{\tau}\delta_{t}\phi^{n+1}+\frac{1}{2\tau}\delta_{tt}\phi^{n+1}
$, $\hat{\phi}^{n+\frac{1}{2}}
:=\frac{3}{2}\phi^{n}-\frac{1}{2}\phi^{n-1}$ and $\hat{\phi}^{n+1}
:=2\phi^{n}-\phi^{n-1}$. Following identities will be used frequently as well
\begin{equation}\label{eq:ID:1}
2(h^{n+1}-h^n, h^{n+1}) = \|h^{n+1}\|^2 - \|h^n\|^2 + \|h^{n+1}-h^n\|^2,
\end{equation}
\begin{equation}\label{eq:ID:2}
(D_\tau h^{n+1}, h^{n+1}) = \frac{1}{4\tau}(\|h^{n+1}\|^2+\|2h^{n+1}\!\!-\!h^n\|^2
-\|h^{n}\|^2\!-\!\|2h^{n}\!-\!h^{n-1}\|^2
+\|\delta_{tt}h^{n+1}\|^2).
\end{equation}
To prove energy stability of the numerical schemes, we
assume that the derivative of $f$ in equation \eqref{eq:CH}
is uniformly bounded, i.e.
\begin{equation}\label{eq:Lip}
\max_{\phi\in\mathbf{R}} | f'(\phi) | \le L,
\end{equation}
where $L$ is a non-negative constant.

Note that, if the phase-field system satisfies the maximum principle, then \eqref{eq:Lip} is satisfied for any smooth $f$.
Although the Cahn-Hilliard equation does not satisfy the maximum principle, it has been shown that
in \cite{caffarelli_l_1995} that for a truncated potential $F$ with quadratic growth at infinities, the maximum norm of
the solution to the Cahn-Hilliard equation is bounded. 
On the other hand, for a more general potential $F$, Feng and Prohl \cite{feng_numerical_2005} proved that if the Cahn-Hilliard equation convergence to 
its sharp-interface limit, then its solution has a $L^\infty$ bound.
Therefore, it has been a common practice
(cf. \cite{kessler_posteriori_2004,shen_numerical_2010,condette_spectral_2011}) to consider the Cahn-Hilliard equations with a truncated double-well potential $F$ such that \eqref{eq:Lip} is satisfied.

\subsection{The stabilized linear Crank-Nicolson scheme}

Suppose $\phi^0=\phi_0(\cdot)$ and
$\phi^1\approx \phi(\cdot,\tau)$ are given, our stabilized liner Crank-Nicolson scheme (abbr. SL-CN)
calculates $\phi^{n+1}, n=1,2,\ldots,N=T/\tau-1$
iteratively, using
\begin{gather}\label{eq:CN:1}
  \frac{\phi^{n+1}-\phi^{n}}{\tau}=\gamma\Delta \mu^{n+\frac{1}{2}},\\
\label{eq:CN:2}
  \mu^{n+\frac{1}{2}}=-\varepsilon \Delta \Big(\frac{\phi^{n+1}+\phi^{n}}{2} \Big)
  +\frac{1}{\varepsilon}f\Big(\frac{3}{2}\phi^{n}-\frac{1}{2}\phi^{n-1}\Big)
  -A\tau \Delta \delta_{t}\phi^{n+1}
  +B\delta_{tt}\phi^{n+1},
\end{gather}
where $A$ and $B$ are two non-negative constants to
stabilize the scheme.

\begin{theorem}\label{cn}
 Assume \eqref{eq:Lip} is satisfied.
  Under the condition
\begin{equation}\label{eq:CN:ABcond}
  A\geq \dfrac{L^{2}}{16\varepsilon^{2}}\gamma, \quad  
  B\geq \dfrac{L}{2\varepsilon},
\end{equation}
the following energy dissipation law
\begin{equation}\label{eq:CN:Edis}
  E_{C}^{n+1}\leq E_{C}^{n}-\Big(2\sqrt{\frac{A}{\gamma}}-\frac{L}{2\varepsilon}\Big)\|\delta_{t}\phi^{n+1}\|^{2}
  -\Big(\frac{B}{2}-\frac{L}{4\varepsilon}\Big)\|\delta_{tt}\phi^{n+1}\|^{2},\hspace{1cm} \forall n\geq1,
\end{equation}
holds for the scheme (\ref{eq:CN:1})-(\ref{eq:CN:2}), where
\begin{equation}\label{eq:CN:E}
  E_{C}^{n+1}=E_{\varepsilon}(\phi^{n+1})
  +\Big(\frac{L}{4\varepsilon}+\frac{B}{2}\Big)\|\delta_{t}\phi^{n+1}\|^{2}.
\end{equation}
\end{theorem}

\begin{proof}
  Pairing \eqref{eq:CN:1} with
  $\tau \mu^{n+\frac{1}{2}}$, \eqref{eq:CN:2} with
  $-\delta_t\phi^{n+1}$, and combining the results, we get
  \begin{equation}\label{cn5}
    \begin{split}
      &\frac{\varepsilon}{2}(\|\nabla \phi^{n+1}\|^{2} -
      \|\nabla \phi^{n}\|^{2})
      +\frac{1}{\varepsilon}(f\big(\hat{\phi}^{n+\frac{1}{2}}\big),\delta_t\phi^{n+1} )\\
      =&-\gamma\tau \|\nabla \mu^{n+\frac{1}{2}}\|^{2}-A\tau \|\nabla
      \delta_{t}\phi^{n+1}\|^{2}
      -B(\delta_{tt}\phi^{n+1},\delta_t\phi^{n+1}).
    \end{split}
  \end{equation}
  Pairing (\ref{eq:CN:1}) with
  $2\sqrt{\frac{A}{\gamma}}\tau\delta_t\phi^{n+1}$, then using
  Cauchy-Schwartz inequality, we get
  \begin{equation}\label{cn7}
  2\sqrt{\tfrac{A}{\gamma}}\|\delta_{t} \phi^{n+1}\|^{2}
  = -2\sqrt{A\gamma}\tau(\nabla \mu^{n+\frac12},\nabla\delta_t \phi^{n+1})
  \leq \gamma\tau \|\nabla \mu^{n+\frac12}\|^{2} +
  A\tau\|\nabla \delta_{t}\phi^{n+1}\|^{2}.
  \end{equation}
  To handle the term involving $f$,  we expand $F(\phi^{n+1})$ and $F(\phi^n)$ at
  $\hat{\phi}^{n+\frac{1}{2}}$ as
  \begin{align*}
    F(\phi^{n+1}) &= F(\hat{\phi}^{n+\frac{1}{2}})+f(\hat{\phi}^{n+\frac{1}{2}})(\phi^{n+1}-\hat{\phi}^{n+\frac{1}{2}})+\frac{1}{2}f'(\xi^{n}_{1})(\phi^{n+1}-\hat{\phi}^{n+\frac{1}{2}})^{2},\\
    F(\phi^{n}) &= F(\hat{\phi}^{n+\frac{1}{2}})+f(\hat{\phi}^{n+\frac{1}{2}})(\phi^{n}-\hat{\phi}^{n+\frac{1}{2}})+\frac{1}{2}f'(\xi^{n}_{2})(\phi^{n}-\hat{\phi}^{n+\frac{1}{2}})^{2},
  \end{align*}
  where $\xi^{n}_1$ is a number between $\phi^{n+1}$ and $\hat{\phi}^{n+\frac12}$, $\xi^{n}_2$ is a number between $\phi^n$ and $\hat{\phi}^{n+\frac12}$.
  Taking the difference of above two equations, we have
  \begin{equation}\nonumber
    \begin{split}
      &  F(\phi^{n+1})-F(\phi^{n}) - f(\hat{\phi}^{n+\frac{1}{2}})(\phi^{n+1}-\phi^{n})\\
      ={} &  \frac{1}{2}f'(\xi^{n}_{1})
      \left[(\phi^{n+1}-\hat{\phi}^{n+\frac{1}{2}})^{2} - (\phi^{n}-\hat{\phi}^{n+\frac{1}{2}})^{2}
      \right]
      - \frac{1}{2}(f'(\xi^{n}_{2})-f'(\xi^{n}_{1}))(\phi^{n}-\hat{\phi}^{n+\frac{1}{2}})^{2}\\
      ={} &  
      \frac{1}{2}f'(\xi^{n}_{1})\delta_{t}\phi^{n+1}\delta_{tt}\phi^{n+1}
      - \frac{1}{8}(f'(\xi^{n}_{2})-f'(\xi^{n}_{1}))(\delta_{t}\phi^{n})^{2}\\
      \le{} & 
      \frac{L}{4}(|\delta_t \phi^{n+1}|^2 + |\delta_{tt}\phi^{n+1}|^2)
      + \frac{L}{4}|\delta_t\phi^n|^2.
    \end{split}
  \end{equation}
  Multiplying the above equation with $\dfrac{1}{\varepsilon}$, then taking integration leads to
 \begin{equation}\label{eq:cn2}
   \frac{1}{\varepsilon}(F(\phi^{n+1})-F(\phi^{n})-f(\hat{\phi}^{n+\frac{1}{2}})\delta_t\phi^{n+1}, 1) 
  \le
  \frac{L}{4\varepsilon}(\|\delta_t \phi^{n+1}\|^2 + \|\delta_{tt}\phi^{n+1}\|^2
  + \|\delta_t\phi^n\|^2).
  \end{equation}
For the term involving $B$, by using identity
\eqref{eq:ID:1} with $h^{n+1}=\delta_t \phi^{n+1}$, one get
\begin{equation}\label{cs11-0}
  -B(\delta_{tt}\phi^{n+1},\delta_t\phi^{n+1})
  =-\frac{B}{2}\|\delta_{t}\phi^{n+1}\|^{2}+\frac{B}{2}\|\delta_{t}\phi^{n}\|^{2}
  -\frac{B}{2}\|\delta_{tt}\phi^{n+1}\|^{2}.
\end{equation}
Summing up \eqref{cn5}-\eqref{cs11-0}, we obtain
\begin{equation}\label{cn8}
  \begin{split}
    &\frac{\varepsilon}{2}(\|\nabla \phi^{n+1}\|^{2} -
    \|\nabla \phi^{n}\|^{2}) 
    + \frac{1}{\varepsilon}(F(\phi^{n+1})-F(\phi^{n}),1)
    +
    \frac{B}{2}\|\delta_{t}\phi^{n+1}\|^{2} -\frac{B}{2}\|\delta_{t}\phi^{n}\|^{2}\\
    \leq & -
    2\sqrt{\frac{A}{\gamma}}\|\delta_{t}\phi^{n+1}\|^{2}
    +\frac{L}{4\varepsilon}\|\delta_{t} \phi^{n+1}\|^{2}
    +\frac{L}{4\varepsilon}\|\delta_{t} \phi^{n}\|^{2}
    -\frac{B}{2}\|\delta_{tt}\phi^{n+1}\|^{2}
    +\frac{L}{4\varepsilon}\|\delta_{tt} \phi^{n+1}\|^{2},
\end{split}
\end{equation}
which is the energy estimate \eqref{eq:CN:Edis}.\qed
\end{proof}

\subsection{The stabilized linear BDF2 scheme}

Suppose $\phi^0=\phi_0(\cdot)$ and
$\phi^1\approx \phi(\cdot,\tau)$ are given, our stabilized
linear BDF2 scheme (abbr. SL-BDF2) calculate
$\phi^{n+1}, n=1,2,\ldots,N=T/\tau-1$ iteratively, using
\begin{gather}\label{eq:BDF:1}
  \frac{3\phi^{n+1}-4\phi^{n}+\phi^{n-1}}{2\tau}=\gamma\Delta \mu^{n+1},\\
  \mu^{n+1}=-\varepsilon \Delta \phi^{n+1} + \frac{1}{\varepsilon}f(2\phi^{n}-\phi^{n-1})
  -A\tau \Delta \delta_{t}\phi^{n+1}
  +B\delta_{tt}\phi^{n+1}, \label{eq:BDF:2}
\end{gather}
where $A$ and $B$ are two non-negative constants.

\begin{theorem}\label{ac1}
 Assume \eqref{eq:Lip} is satisfied, and
  \begin{equation}\label{eq:BDF:initm}
	 \frac{1}{|\Omega|} \int_\Omega \phi^1 dx=\frac{1}{|\Omega|}\int_\Omega \phi^0 dx=m_0.
  \end{equation}
  Then under the condition
  \begin{equation}\label{eq:BDF:ABcond}
    B\geq\dfrac{L}{\varepsilon};
    \quad
    A\geq\frac{\gamma}{\alpha_2}\dfrac{L^{2}}{16\varepsilon^{2}} -\alpha_1\frac{\varepsilon}{2\tau}, 
    \quad  0\le\alpha_1\le 1,
    \quad  0<\alpha_2\le1,
  \end{equation}
  the following energy dissipation law
  \begin{multline}\label{eq:BDF:Edis}
      E_{B}^{n+1} \leq E_{B}^{n} -\frac{1}{4\tau\gamma}\|\delta_{tt}\phi^{n+1}\|_{-1}^{2}
      -(1-\alpha_1)\frac{\varepsilon}{2}\|\nabla\delta_{t}\phi^{n+1}\|^{2}
      -(1-\alpha_2)\frac{1}{\tau\gamma}\|\delta_{t}\phi^{n+1}\|^{2}_{-1}\\
      -\left(2\sqrt{\frac{\alpha_2}{\gamma}(A+\frac{\alpha_1\varepsilon}{2\tau})}-\frac{L}{2\varepsilon}\right)\|\delta_{t}\phi^{n+1}\|^{2}
      -\Big(\frac{B}{2}-\frac{L}{2\varepsilon}\Big)\|\delta_{tt}\phi^{n+1}\|^{2},
      \quad \forall\: n\geq1,
  \end{multline}
  holds for the scheme (\ref{eq:BDF:1})-(\ref{eq:BDF:2}),
  where 
  \begin{equation}\label{eq:BDF:E}
    E_{B}^{n+1}=E_{\varepsilon}(\phi^{n+1})+\frac{1}{4\tau\gamma}\|\delta_{t}\phi^{n+1}\|_{-1}^{2}
    +\Big(\frac{L}{2\varepsilon}+\frac{B}{2}\Big)
    \|\delta_{t}\phi^{n+1}\|^{2}.
  \end{equation}
\end{theorem}

\begin{proof}
	1) Integration both sides of equation \eqref{eq:BDF:1}, then using the Neumann boundary condition of $\mu$ and property \eqref{eq:BDF:initm}, we get
	\begin{equation}
	\frac{1}{|\Omega|}\int_\Omega \phi^{n+1} dx = m_0,\quad n=1,\ldots, N.
	\end{equation}
	Thus $\delta_t\phi^{n+1} \in L_0^2$ for $n=0,\ldots, N$.
	
  2) Pairing (\ref{eq:BDF:1}) with
  $(-\Delta)^{-1}\delta_t\phi^{n+1}/\gamma$, plus (\ref{eq:BDF:2}) paired with $-\delta_t\phi^{n+1}$, we get
  \begin{equation}\label{cs7}
    \begin{split}
      \Big(\frac{1}{\gamma}D_\tau\phi^{n+1},(-\Delta)^{-1}\delta_t\phi^{n+1}\Big)
      ={} & \varepsilon (\Delta \phi^{n+1},\delta_t\phi^{n+1})
      -\frac{1}{\varepsilon}(f(\hat{\phi}^{n+1}),\delta_t\phi^{n+1})\\
      & - A\tau \|\nabla
      \delta_{t}\phi^{n+1} \|^2
      -B(\delta_{tt}\phi^{n+1},\delta_t\phi^{n+1}).
    \end{split}
  \end{equation}
By integration by parts and applying \eqref{eq:ID:1}, \eqref{eq:ID:2}, following identities hold
  \begin{equation}\label{cs8}
    \begin{split}
      &-\Big(\frac1\gamma D_\tau\phi^{n+1},(-\Delta)^{-1}\delta_t\phi^{n+1}\Big)\\
      = &-\frac{1}{\tau\gamma}\|\delta_{t}\phi^{n+1}\|_{-1}^{2}
      -\frac{1}{4\tau\gamma}\left(\|\delta_{t}\phi^{n+1}\|_{-1}^{2}-\|\delta_{t}\phi^{n}\|_{-1}^{2}
      +\|\delta_{tt}\phi^{n+1}\|_{-1}^{2}\right),
    \end{split}
  \end{equation}
  \begin{equation}\label{cs9}
    \varepsilon (\Delta \phi^{n+1},\delta_t\phi^{n+1})
    =-\frac{\varepsilon}{2}(\|\nabla\phi^{n+1}\|^{2}-\|\nabla\phi^{n}\|^{2}
    +\|\nabla\delta_{t}\phi^{n+1}\|^{2}),
  \end{equation}
  \begin{equation}\label{cs11-1}
    -B(\delta_{tt}\phi^{n+1},\delta_t\phi^{n+1})
    =-\frac{B}{2}\|\delta_{t}\phi^{n+1}\|^{2}+\frac{B}{2}\|\delta_{t}\phi^{n}\|^{2}
    -\frac{B}{2}\|\delta_{tt}\phi^{n+1}\|^{2}.
  \end{equation}
  To handle the term involves $f$ in \eqref{cs7}, we
  expand $F(\phi^{n+1})$ and $F(\phi^n)$ at
  $\hat{\phi}^{n+1}$ as
  \begin{align*}\nonumber
    F(\phi^{n+1})&=F(\hat{\phi}^{n+1})+f(\hat{\phi}^{n+1})(\phi^{n+1}-\hat{\phi}^{n+1})+\frac{1}{2}f'(\zeta^{n}_{1})(\phi^{n+1}-\hat{\phi}^{n+1})^{2},\\
    F(\phi^{n})&=F(\hat{\phi}^{n+1})+f(\hat{\phi}^{n+1})(\phi^{n}-\hat{\phi}^{n+1})+\frac{1}{2}f'(\zeta^{n}_{2})(\phi^{n}-\hat{\phi}^{n+1})^{2},
  \end{align*}
 where $\zeta^{n}_1$ is a number between $\phi^{n+1}$ and $\hat{\phi}^{n+1}$, $\zeta^{n}_2$ is a number between $\phi^n$ and $\hat{\phi}^{n+1}$.
 Taking the difference of above two equations, using the fact 
  $\phi^{n+1}-\hat{\phi}^{n+1}=\delta_{tt}\phi^{n+1}$ and $\phi^{n}-\hat{\phi}^{n+1} = -\delta_t\phi^{n}$,
  we obtain
  \begin{equation}\label{bta4}
    \begin{split}
      F(\phi^{n+1})-F(\phi^{n}) 
      -f(\hat{\phi}^{n+1})\delta_t\phi^{n+1} 
      ={}&
      \frac{1}{2}f'(\zeta^{n}_{1})(\delta_{tt}\phi^{n+1})^{2}
      -\frac{1}{2}f'(\zeta^{n}_{2})(\delta_{t}\phi^{n})^{2}\\
      \le{} &
      \frac{L}{2}|\delta_{tt}\phi^{n+1}|^2
      +\frac{L}{2}|\delta_{t}\phi^{n}|^2.
    \end{split}
  \end{equation}
  Taking inner product of the above equation with constant $1/\varepsilon$, then combining the result with 
   (\ref{cs7}), (\ref{cs8}), (\ref{cs9}) and (\ref{cs11-1}),
   we obtain
  \begin{equation}\label{cs12}
    \begin{split}
      &\frac{1}{\varepsilon}(F(\phi^{n+1})-F(\phi^{n}),1)
      +\frac{\varepsilon}{2}(\|\nabla\phi^{n+1}\|^{2}-\|\nabla\phi^{n}\|^{2})\\
      &+\frac{1}{4\tau\gamma}(\|\delta_{t}\phi^{n+1}\|_{-1}^{2}-\|\delta_{t}\phi^{n}\|_{-1}^{2})
      +\Big(\frac{L}{2\varepsilon}+\frac{B}{2}\Big)(\|\delta_{t}\phi^{n+1}\|^{2}-\|\delta_{t}\phi^{n}\|^{2})\\
      \leq&
      -\frac{1}{4\tau\gamma}\|\delta_{tt}\phi^{n+1}\|_{-1}^{2}
      -\frac{1}{\tau\gamma}\|\delta_{t}\phi^{n+1}\|_{-1}^{2}
      -\frac{\varepsilon}{2}\|\nabla\delta_{t}\phi^{n+1}\|^{2}
      -A\tau\|\nabla\delta_{t}\phi^{n+1}\|^{2}\\
      &+\frac{L}{2\varepsilon}\|\delta_{t}\phi^{n+1}\|^{2} 
      -\frac{B}{2}\|\delta_{tt}\phi^{n+1}\|^{2}+\frac{L}{2\varepsilon}\|\delta_{tt}\phi^{n+1}\|^{2}.
    \end{split}
  \end{equation}
  Combining the above equation and the inequality
  \begin{equation}\label{cs14}
    \chi\|\nabla\delta_{t}\phi^{n+1}\|^{2}+\dfrac{\alpha_2}{\tau\gamma}\|\delta_{t}\phi^{n+1}\|_{-1}^{2}
    \geq 2\sqrt{ \frac{\chi\alpha_2}{\tau\gamma}}\|\delta_{t}\phi^{n+1}\|^{2},
  \end{equation}
  with $\chi=A\tau+ \frac{\alpha_1\varepsilon}{2}, 0\le \alpha_1\le 1$, $0<\alpha_2\le1$, we get the energy estimate \eqref{eq:BDF:Edis}.\qed
\end{proof}

\begin{remark}\label{rmk:stab3}
	The discrete Energy $E_B$ defined in equation
	\eqref{eq:BDF:E} is a second order approximations to the
	original energy $E$, since
	$ \| \delta_t \phi^{n+1} \|^2, \| \delta_t \phi^{n+1} \|_{-1}^2 \sim O(\tau^2)$ provided that the schemes converge.  On the
	other side, summing up the equation \eqref{eq:BDF:Edis} with $\alpha_1=\alpha_2=1$ for
	$n=1,\ldots, N$, we get
	\begin{equation}\label{eq:steady}
	E^{N+1}_{B} + \sum_{n=1}^N \left( 
	\frac{1}{4\tau\gamma}\|\delta_{tt}\phi^{n+1}\|_{-1}^{2}
	+\beta_1\|\delta_{t}\phi^{n+1}\|^{2}
	+\beta_2\|\delta_{tt}\phi^{n+1}\|^{2}
	\right) \leq E^1_{B},
	\end{equation}
	where $\beta_1=2\sqrt{\frac{A}{\gamma}+\frac{\varepsilon}{2\tau\gamma}}-\frac{L}{2\varepsilon}$, $\beta_2=\frac{B}{2}-\frac{L}{2\varepsilon}$.
	By taking $N\rightarrow\infty$, we get
	$\delta_t \phi^{N+1} \rightarrow 0$ and $\delta_{tt} \phi^{N+1} \rightarrow 0$ if $\beta_1>0$ and $\beta_2\ge0$, which means the
	system will eventually converge to a steady state. By
	equation \eqref{eq:BDF:1} and \eqref{eq:BDF:2}, this
	steady state is an extreme point of the original energy
	functional $E$. Same argument applies to the LS-CN scheme
	and similar second order stabilization schemes for the Allen-Cahn equation~\cite{wang_energy_2018}.
\end{remark}

\begin{remark}\label{rmk:stab1}
  The constant $A$ defined in equation \eqref{eq:CN:ABcond}
  and \eqref{eq:BDF:ABcond} seems to be quite large when
  $\varepsilon$ is small, but it is not necessarily
  true. Since usually $\gamma$ is a small constant related
  to $\varepsilon$. For example, it was showed in
  \cite{magaletti2013sharp} that, the Cahn-Hilliard equation
  coupled with the Navier-Stokes equations have a
  sharp-interface limit when
  $O(\varepsilon^3) \le \gamma \le O(\varepsilon)$, while
  $\gamma \sim O(\varepsilon^2)$ gives the fastest
  convergence. A similar result is obtained for the Cahn-Hilliard Navier-Stokes system with a more general boundary condition~\cite{xu_sharp-interface_2017}. On the other hand, the numerical results in
  Section 4 shows $A$ can take much smaller values than
  those defined in \eqref{eq:CN:ABcond} and
  \eqref{eq:BDF:ABcond} when nonzero $B$ values are used.
\end{remark}

\begin{remark}\label{rmk:stab2}
	From equation \eqref{eq:BDF:ABcond}, we see that 
	the SL-BDF2 scheme is stable (equation \eqref{eq:BDF:Edis} holds with $\alpha_1=\alpha_2=1$) with any $A \ge 0$, if
	\begin{equation}\nonumber
	 \tau \leq \frac{8\varepsilon^3}{L^2 \gamma}.
	\end{equation}
	If $\tau$ small enough, the combination of the term
	$\frac{1}{\tau}\|\delta_{t}\phi^{n+1}\|_{-1}^{2}$ and term
	$\frac{\varepsilon}{2}\|\nabla\delta_{t}\phi^{n+1}\|^{2}$
	controls the term $\| \delta_{tt} \phi^{n+1} \|^2$ as
	well, since 
	\begin{equation}\nonumber
	 \| \delta_{tt} \phi^{n+1} \|^2 \leq 2\| \delta_t \phi^{n+1} \|^2
	  + 2\| \delta_t \phi^n \|^2.
	\end{equation}
	 A direct calculation shows that the SL-BDF2 scheme is stable with
    any $A\ge 0$, $B\ge 0$, if
	\begin{equation}\label{eq:BDF:tcond}
	 \tau \leq \frac{8\varepsilon^3}{25L^2 \gamma}.
	\end{equation}
\end{remark}


\section{Convergence analysis}

In this section, we shall establish the error estimate of
the semi-discretized scheme SL-BDF2 for the Cahn-Hilliard
equation in the norm of $l^{\infty}(0,T;H^{-1})\cap l^{2}(0,T;H^{1})$. We will shown that, if the interface is well
developed in the initial condition, the error bounds depend
on $\frac{1}{\varepsilon}$ only in some lower polynomial
order for small $\varepsilon$. Similar error estimate result can be obtained for the SL-CN scheme but the analysis is more involved, we put it into  \cite{WangYu2017b}.

Let $\phi(t^{n})$ be the
exact solution at time $t=t^{n}$ to equation
(\ref{eq:CH}) and $\phi^{n}$ be the solution at time
$t=t^{n}$ to the time discrete numerical scheme
(\ref{eq:BDF:1})-(\ref{eq:BDF:2}), we define error function
$e^{n}:=\phi^{n}-\phi(t^{n})$. Obviously $e^{0}=0$.
 
Before presenting the detailed error analysis, we first make
some assumptions. For simplicity, we take $\gamma=1$ in this
section, and assume $0<\varepsilon \ll 1$.  We use notation
$\lesssim$ in the sense that $f\lesssim g$ means that
$f \le C g$ with a positive constant $C$ independent of
$\tau, \varepsilon$.
\begin{assump}\label{ap:1}
  We assume that $f$ either satisfies the following properties (i) and (ii), or (i) and (iii).
  \begin{enumerate}
  \item [(i)]$F\in C^{4}(\mathbf{R})$, 
    $F(\pm 1)=0$, and $F>0$ elsewhere. There exist two
    non-negative constants $B_0,B_1$, such that
    \begin{equation}\label{eq:AP:Fcoercive} \phi^2 \le
      B_0 + B_1 F(\phi),\quad \forall\; \phi\in\mathbf{R}.
	\end{equation}
  \item[(ii)] $f=F'$. $f'$ and $f''$ are uniformly bounded, or,  $f$
    satisfies \eqref{eq:Lip} and
    \begin{equation}\label{eq:Lip2}
      \max_{\phi\in\mathbf{R}} | f''(\phi) | \le L_2,
    \end{equation}
    where $L_2$ is a non-negative constant.
   \item[(iii)] $f$  satisfies for some finite $2\le p \leq 
    3+ \frac{d}{3(d-2)}$
    and positive numbers $\tilde{c}_{i} >0$, $i=0,\ldots,5$,
    \begin{equation}\label{eq:AP:fp}
    \tilde{c}_1 |\phi|^{p-2} - \tilde{c}_0 \leq f'(\phi) \leq \tilde{c}_2 |\phi|^{p-2}+\tilde{c}_3,
    \end{equation}
    \begin{equation}\label{eq:AP:fpp}
    | f''(\phi) | \leq \tilde{c}_4 |\phi|^{(p-3)^+}+\tilde{c}_5,
    \end{equation}
    where for any real number $a$, the notation $(a)^+ := \max\{ a, 0\}$.\qed
  \end{enumerate}
\end{assump}

Note that Assumption \ref{ap:1} (ii) is a special case of Assumption \ref{ap:1} (iii) with $p=2$. The commonly-used quartic double-well potential satisfies Assumption (i) and (iii) with $p=4$.
Furthermore, from equation \eqref{eq:AP:fp} we easily get 
\begin{equation} \label{eq:AP:fpl}
-(f'(\phi) u, u) \le \tilde{c}_0 \| u \|^2, \quad \forall\, u\in L^2(\Omega).
\end{equation}
\begin{assump}\label{ap:2}
	We assume 
	that $\phi^0$ is smooth enough. More precisely,
	there exist constant $m_{0}$ and
	non-negative constants $\sigma_{1}, \ldots, \sigma_6$, such that
	\begin{equation}\label{eq:AP:m0}
	m_{0}:=\frac{1}{|\Omega|}\int_{\Omega}\phi^{0}(x){\rm d}x \in (-1,1),
	\end{equation}	
	\begin{equation}\label{eq:AP:E0}
	E_{\varepsilon}(\phi^{0}):=\frac{\varepsilon}{2}\|\nabla \phi^{0}\|^{2}+\frac{1}{\varepsilon}\|F(\phi^{0})\|_{L^{1}}\lesssim \varepsilon^{-\sigma_{1}},
	\end{equation}
	\begin{equation}\label{eq:AP:htHn1}
	\|\phi_t^0\|_{-1}^2\lesssim \varepsilon^{-\sigma_2},
	\end{equation}
	\begin{equation}\label{eq:AP:htL2}
	\|\phi_{t}^0\|^2\lesssim \varepsilon^{-\sigma_3};
	\end{equation}
	\begin{equation}\label{eq:AP:LEt0}
	\varepsilon\|\nabla \phi_{t}^0\|^2 
	+ \frac{1}{\varepsilon} (f'(\phi^0)\phi_t^0,\phi_t^0) 
	\lesssim \varepsilon^{-\sigma_{4}},
	\end{equation}
	\begin{equation}\label{eq:AP:httHn2}
	\|\Delta^{-1}\phi_{tt}^0\|^2
	\lesssim \varepsilon^{-\sigma_{5}},
	\end{equation}
	\begin{equation}\label{eq:AP:httHn1}
	\|\phi_{tt}^0\|_{-1}^2
	\lesssim \varepsilon^{-\sigma_{6}}.
	\end{equation}\qed
\end{assump}

Given Assumption \ref{ap:1} (i)(iii) and Assumption \ref{ap:2}, we have
following estimates for the exact solution to the Cahn-Hilliard equation.
The proof is given in Appendix A. 

\begin{lemma}\label{lm:reg} 
	Suppose Assumption \ref{ap:1} (i)(iii) and Assumption \ref{ap:2} are satisfied. We have following regularity results for the exact solution $\phi$ of (\ref{eq:CH}) with $\gamma = 1$.
	\begin{enumerate}
		\item[(i)]
		$\int_{0}^{\infty}\|\phi_t\|_{-1}^{2}{\rm d}t + \esssup \limits_{t\in[0,\infty]}
		E_{\varepsilon}(\phi) \lesssim \varepsilon^{-\rho_1}$,
		and $\| \phi \|_1^2 \lesssim \varepsilon^{-(\sigma_1+1)}$;
		
		\item[(ii)]
		$   \esssup \limits_{t\in[0,\infty]}  \|\phi_{t}\|_{-1}^2+ \varepsilon\int_{0}^{\infty} \|\nabla \phi_{t}\|^2{\rm d} t
		\lesssim \varepsilon^{-\rho_2} $;
		
		\item[(iii)]
		$\esssup \limits_{t\in[0,\infty]}\|\phi_t\|^2 + \varepsilon \int_{0}^{\infty} \|\Delta \phi_t\|^2 {\rm d} t
		\lesssim \varepsilon^{- \rho_3}$;
		
		\item[(iv)]
		$ \int_{0}^{\infty} \|\phi_{tt}\|_{-1}^2 {\rm  d} t
		+  \esssup \limits_{t\in[0,\infty]} \varepsilon\|\nabla \phi_{t}\|^2
		\lesssim \varepsilon^{-\rho_4}$;
		
		\item[(v)]
		$\esssup \limits_{t\in[0,\infty]}  \| \Delta^{-1}\phi_{tt}\|^2 + \varepsilon\int_{0}^{\infty} \| \phi_{tt} \|^2 {\rm d} t
		\lesssim \varepsilon^{-\rho_5}$;
		
		\item[(vi)]
		$\int_{0}^{\infty} \|\Delta^{-1}\phi_{ttt}\|_{-1}^2 {\rm d} t
		+ \esssup \limits_{t\in[0,\infty]} \varepsilon \| \phi_{tt}\|_{-1}^2 \lesssim \varepsilon^{-\rho_6}$;
	\end{enumerate}
	where $\rho_1=\sigma_1$ and
	\begin{equation*}
	\begin{split}
	\rho_2&=\max\{\sigma_1+3, \sigma_2\}, \\
	\rho_3&=\max\{ (\sigma_1+1)(p-2)+\rho_2+4, \sigma_3\}, \\
	\rho_4&=\max\{ \rho_2+2+\tfrac12\rho_3+\tfrac12(\sigma_1+1)(p-3)^+, \sigma_4\}, \\
	\rho_5&=\max\{ \rho_2+\rho_4+1+(\sigma_1+1)(p-3)^+, (\sigma_1+1)(p-2)+\rho_4+3, \sigma_5\}, \\
	\rho_6&=\max\{ (\sigma_1+1)(p-2)+\rho_5+3, \sigma_6-1\}.
	\end{split}
	\end{equation*}	\qed
	
\end{lemma}

To get the convergence result of the second order schemes, we need make some assumptions on the 
scheme used to calculate the numerical solution at first time step. 

\begin{assump}\label{ap:3}
We assume that an appropriate scheme is used to
calculate the numerical solution at first step, such
that
\begin{equation}\label{eq:AP:m1}
m_1:=\frac{1}{|\Omega|}\int_{\Omega}\phi^{1}(x){\rm d}x = m_0,
\end{equation}
\begin{equation}\label{eq:AP:phi1}
  E_\varepsilon (\phi^1) \le E_\varepsilon (\phi^0) \lesssim\varepsilon^{-\sigma_1 },
\end{equation}
\begin{equation} \label{eq:AP:phi1Hn1} 
  \frac{1}{\tau}\| \phi^1
  - \phi^0 \|^2_{-1}  \lesssim
  \varepsilon^{-\sigma_1 },
\end{equation}
\begin{equation} \label{eq:AP:phi1L2} 
\frac{1}{\tau}\|\phi^1-\phi^0\|^2 \lesssim\varepsilon^{-\sigma_1 -2},
\end{equation}
and exist a constant $0<\tilde{\sigma}_1<\max\{\rho_6 +4,\rho_4+6,\rho_2+9\}$ and
constant $C_1$ independent of $\tau,\varepsilon$, such that
\begin{equation}\label{eq:AP:phi1e}
	\| e^1 \|^2_{-1} + \tau\varepsilon\| \nabla e^1 \|^2 \le C_1 \varepsilon^{-\tilde{\sigma}_1}\tau^4.
\end{equation}\qed
\end{assump}


Following volume conservation property is easy to prove but 
important to the error estimate. Because of
the integration of $\phi^n$ is conserved, $\delta_t \phi^n$ and $e^n$ belong to $L_0^2(\Omega)$ such that we can define $H^{-1}$ norm and
use Poincare's inequality for those quantities.

\begin{lemma}\label{lm:stab}
	Suppose \eqref{eq:AP:m0} and \eqref{eq:AP:m1} holds, then
	the numerical solution of \eqref{eq:BDF:1}-\eqref{eq:BDF:2}
	satisfies
	\begin{equation}\label{eq:CA:conserv}
	\frac{1}{|\Omega|}\int_{\Omega}\phi^{n}(x){\rm d}x=m_0,\quad n=1,\ldots, N+1, 
	\end{equation}
	and the error function $e^n$ satisfies
	\begin{equation}\label{eq:CA:e0}
	\int_{\Omega}e^{n}(x){\rm d}x=0,\quad n=1,\ldots, N+1. 
	\end{equation}\qed
\end{lemma}

Now, we present our first error estimate result, which is a coarse estimate obtained by a standard approach.
\begin{theorem}(Coarse error estimate)\label{prio}
  Suppose Assumption \ref{ap:1} (i)(ii), Assumption \ref{ap:2} and Assumption \ref{ap:3} hold. Then
  $\forall\, \tau \leq1$, following error estimate holds for the 
   SL-BDF2 scheme \eqref{eq:BDF:1}-\eqref{eq:BDF:2}:
  \begin{equation}\label{eq:ES:coarse1}
    \begin{split}
      &\|e^{n+1}\|_{-1}^{2} +\|2e^{n+1}-e^{n}\|_{-1}^{2}
      +2A\tau^{2}\|\nabla e^{n+1}\|^{2}\\
      &+2A\tau^2\|\delta_{t}\nabla e^{n+1}\|^{2}
      	+\tau\varepsilon\|\nabla e^{n+1}\|^{2}
      	+\|\delta_{tt}e^{n+1}\|_{-1}^{2}
      	+4B\tau\|e^{n+1}\|^2\\
      \leq{} &
      \|e^{n}\|_{-1}^{2}+\|2e^{n}-e^{n-1}\|_{-1}^{2}
      +2A\tau^{2}\|\nabla e^{n}\|^{2}\\
      &+C_{2}\tau\varepsilon^{-3}
      \|2e^{n}-e^{n-1}\|_{-1}^{2}
      +C_{3}\tau^{4}\varepsilon^{-\max\{\rho_6 +1,\rho_4+3,\rho_2+6\}},
      \quad n\ge 1,
    \end{split}
  \end{equation}
  and
    \begin{equation}\label{eq:ES:coarse2}
    \begin{split}
    &\max_{1\le n\le N}\left\{\|e^{n+1}\|_{-1}^{2} +\|2e^{n+1}-e^{n}\|_{-1}^{2}
    +2A\tau^{2}\|\nabla e^{n+1}\|^{2}\right\}\\
    &+\sum_{n=1}^{N}\left(2A\tau^2\|\delta_{t}\nabla e^{n+1}\|^{2}
    +\tau\varepsilon\|\nabla e^{n+1}\|^{2}
    +\|\delta_{tt}e^{n+1}\|_{-1}^{2}
    +4B\tau\|e^{n+1}\|^2\right)\\
    \le &
    \exp(C_{2}\varepsilon^{-3}T)
    \left(C_{3}\varepsilon^{-\max\{\rho_6 +1,\rho_4+3,\rho_2+6\}}
    + C_1(5 + 2A\varepsilon^{-1}\tau)\varepsilon^{-\tilde{\sigma}_1}
    \right)\tau^{4},
    \end{split}
    \end{equation}
  where $C_{2}, C_{3}$ are two constants
  that can be uniformly bounded independent of $\varepsilon$ and $\tau$.  
\end{theorem}

\begin{proof}
  The following equations for the error functions hold:
	  \begin{equation}\label{eq:eoe1}
	  D_{\tau}e^{n+1}
	  =\Delta(\mu^{n+1}-\mu(t^{n+1}))
	  +R_1^{n+1},
	\end{equation}
	  \begin{equation}\label{eq:eoe2}
	  \begin{split}
	  \mu^{n+1}-\mu(t^{n+1})=&-\varepsilon\Delta e^{n+1}+\frac{1}{\varepsilon}[f(2\phi^{n}-\phi^{n-1})-f(\phi(t^{n+1}))]\\
	  &-A\tau\Delta \delta_{t}e^{n+1}
	  +B\delta_{tt}e^{n+1}
	  -A \Delta R_{3}^{n+1} +B R_{2}^{n+1},
	  \end{split}
	  \end{equation}
	  where the residual terms are
	  \begin{align*}
	  R_{1}^{n+1}&=\phi_{t}(t^{n+1})-D_{\tau}\phi(t^{n+1}),\\
	  R_{2}^{n+1}&=\delta_{tt}\phi(t^{n+1})=\phi(t^{n+1})-2\phi(t^{n})+\phi(t^{n-1}),\\
	  R_{3}^{n+1}&=\tau\delta_t\phi(t^{n+1})=\tau(\phi(t^{n+1})-\phi(t^{n})).
	  \end{align*} 	
 Pairing (\ref{eq:eoe1}) with $-\Delta^{-1} e^{n+1}$, adding
	  (\ref{eq:eoe2}) paired with $ -e^{n+1}$, we get
	  \begin{equation}\label{eq:eoee}
	  \begin{split}
	  &(D_{\tau}e^{n+1},-\Delta^{-1}e^{n+1})
	  +\varepsilon\|\nabla e^{n+1}\|^{2}
	  +A\tau(\delta_{t}\nabla e^{n+1},\nabla e^{n+1})\\
	  ={}&(R_{1}^{n+1},-\Delta ^{-1 } e^{n+1}) 
	  -B(R_{2}^{n+1},e^{n+1})
	  -A(\nabla R_{3}^{n+1},\nabla e^{n+1})\\
	  &-B(\delta_{tt}e^{n+1},e^{n+1})
	  -\frac{1}{\varepsilon}\left(f(2\phi^{n}-\phi^{n-1})-f(\phi(t^{n+1})),e^{n+1}\right)\\
	  =&:J_{1}+J_{2}+J_{3}+J_{4}+J_{5}.
	  \end{split}
	  \end{equation}
	  First,  for the terms on the left side of (\ref{eq:eoee}), we have
	  \begin{multline}\label{l1}
	  (D_{t}e^{n+1},-\Delta^{-1}e^{n+1})=
	  \frac{1}{4\tau}(\|e^{n+1}\|_{-1}^{2}+\|2e^{n+1}-e^{n}\|_{-1}^{2})\\
	  -\frac{1}{4\tau}(\|e^{n}\|_{-1}^{2}+\|2e^{n}-e^{n-1}\|_{-1}^{2})
	  +\frac{1}{4\tau}\|\delta_{tt}e^{n+1}\|_{-1}^{2},
	  \end{multline}
	  and
	  \begin{equation}\label{12}
	  A\tau(\delta_{t}\nabla e^{n+1},\nabla e^{n+1})
	  =\frac{1}{2}A\tau(\|\nabla e^{n+1}\|^{2}-\|\nabla e^{n}\|^{2}+\|\delta_{t}\nabla e^{n+1}\|^{2}).
	  \end{equation}
Then, we estimate the terms on the right hand side of \eqref{eq:eoee}.
	\begin{equation}\label{eq:J1}
	\begin{split}
	J_{1}=(R_{1}^{n+1},-\Delta^{-1} e^{n+1})
	\leq
	\frac{1}{\eta_1}\|\Delta^{-1}R_{1}^{n+1}\|_{-1}^{2}
	+ \frac{\eta_1}{4}\|\nabla e^{n+1}\|^{2}.
	\end{split}
	\end{equation}
	\begin{equation}\label{eq:J2}
	\begin{split}
	J_{2}=-B(R_{2}^{n+1},e^{n+1})
	\leq
	\frac{B^{2}}{\eta_1}\|R_{2}^{n+1}\|_{-1}^{2}
	+ \frac{\eta_1}{4}\|\nabla e^{n+1}\|^{2}.
	\end{split}
	\end{equation}
	\begin{equation}\label{eq:J3}
	\begin{split}
	J_{3}&=-A(\nabla R_{3}^{n+1},\nabla e^{n+1})
	\leq \frac{A^{2}}{\eta_1}\|\nabla
	R_{3}^{n+1}\|^{2}
	+ \frac{\eta_1}{4}\|\nabla e^{n+1}\|^{2}.
	\end{split}
	\end{equation}
	\begin{equation}\label{eq:J4}
		\begin{split}
		J_{4}&=-B(\delta_{tt}e^{n+1},e^{n+1})
		=-B(e^{n+1}-(2e^n-e^{n-1}),e^{n+1})\\
		& \leq -B\|e^{n+1}\|^2
		+\frac{B^2}{\eta_1}\|2e^n-e^{n-1}\|^2_{-1} 
		+\frac{\eta_1}{4}\|\nabla e^{n+1}\|^2.
		\end{split}
	\end{equation}
	\begin{equation}\label{eq:J5}
	\begin{split}
	J_{5}&=-\frac{1}{\varepsilon}\left(f(2\phi^{n}-\phi^{n-1})-f(\phi(t^{n+1})),e^{n+1}\right)\\
	& \leq \frac{L}{\varepsilon}\left(|2\phi^{n}-\phi^{n-1}-\phi(t^{n+1})|,|e^{n+1}|\right) \\
	& = \frac{L}{\varepsilon}\left(|2e^{n}-e^{n-1}-\delta_{tt}\phi(t^{n+1})|,|e^{n+1}|\right) \\
	&\leq \frac{L^{2}}{\varepsilon^{2}\eta_1}\|2e^n-e^{n-1}\|_{-1}^{2}
	+\frac{L^{2}}{\varepsilon^{2}\eta_1}\|R_{2}^{n+1}\|_{-1}^{2}
	+ \frac{\eta_1}{2}\|\nabla e^{n+1}\|^{2}.
	\end{split}
	\end{equation}
Combining (\ref{eq:eoee})-(\ref{eq:J5}) together,  yields
	\begin{equation}\label{eq:lm6}
	\begin{split}
	&\frac{1}{4\tau}(\|e^{n+1}\|_{-1}^{2}+\|2e^{n+1}-e^{n}\|_{-1}^{2})
	+\frac{1}{2}A\tau\|\nabla e^{n+1}\|^{2}\\
	 &+\frac{1}{2}A\tau\|\delta_{t}\nabla e^{n+1}\|^{2}
	 +\varepsilon\|\nabla e^{n+1}\|^{2}
	 +\frac{1}{4\tau}\|\delta_{tt}e^{n+1}\|_{-1}^{2}
	+B\|e^{n+1}\|^2\\
	\leq{}& 
	\frac{1}{4\tau}(\|e^{n}\|_{-1}^{2}+\|2e^{n}-e^{n-1}\|_{-1}^{2})
	+\frac{1}{2}A\tau\|\nabla e^{n}\|^{2}\\
	&+\frac{1}{\eta_1}\|\Delta^{-1}R_{1}^{n+1}\|_{-1}^{2}
	+\left(B^2+\frac{L^{2}}{\varepsilon^{2}}\right)\frac{1}{\eta_1}\|R_{2}^{n+1}\|_{-1}^{2}
	+\frac{A^{2}}{\eta_1}\|\nabla
	R_{3}^{n+1}\|^{2}\\
	&+\left(B^2+\frac{L^{2}}{\varepsilon^{2}}\right)\frac{1}{\eta_1}\|2e^n-e^{n-1}\|^2_{-1}
	+ \frac{3}{2}\eta_1\|\nabla e^{n+1}\|^{2}.
	\end{split}
	\end{equation}
	  By using Taylor expansions in integral form, one can get estimates
	  for the residuals 
	  \begin{equation} \label{eq:R1}
	  \|\Delta^{-1}R_{1}^{n+1}\|_{-1}^{2}\leq
	  c_1\tau^{3}\int_{t_{n-1}}^{t_{n+1}}\|\partial_{ttt}\Delta^{-1}\phi(t)\|_{-1}^{2}{\rm
	  	d}t,
	  \end{equation}
	  \begin{equation}\label{eq:R2}
	  \|R_{2}^{n+1}\|_{-1}^{2}\leq
	  c_2\tau^{3}\int_{t_{n-1}}^{t_{n+1}}\|\partial_{tt}\phi(t)\|_{-1}^{2}{\rm d}t,
	  \end{equation}
	  \begin{equation} \label{eq:R3} 
	  \|\nabla R_{3}^{n+1}\|^{2}\leq
	  c_3\tau^{3}\int_{t_{n}}^{t_{n+1}}\|\partial_{t}\nabla
	  \phi(t)\|^{2}{\rm d}t,
	  \end{equation}
	  where $c_1,c_2$ and $c_3$ are three constants.
	  
Taking $\eta_1=\varepsilon/2$ in \eqref{eq:lm6} and combining the residual estimates 
\eqref{eq:R1}-\eqref{eq:R3} and Assumption \ref{lm:reg},\ref{ap:3} with equation \eqref{eq:lm6}, we get estimate
  \eqref{eq:ES:coarse1} with 
      \begin{equation*}
      C_2 = 8L^2 + 8B^2\varepsilon^2, \quad
      C_3 = 8c_1+ 8c_2(L^2+B^2\varepsilon^2) + 8c_3A^2\varepsilon^4.
      \end{equation*}
   By using a discrete Gronwall
    inequality, we obtain \eqref{eq:ES:coarse2}.\qed
\end{proof}

Theorem \ref{prio} is the usual error estimate, in which
the error growth depends on $1/\varepsilon$
exponentially. To obtain a finer estimate on the error, we
need to use a spectral estimate of the linearized
Cahn-Hilliard operator by Chen \cite{chen_spectrum_1994} for
the case when the interface is well developed in the
Cahn-Hilliard system.
\begin{lemma}\label{lem:spectrum}
  Let $\phi(t)$ be the exact solution of Cahn-Hilliard
  equation \eqref{eq:CH} with interfaces are well developed
  in the initial condition (i.e. conditions (1.9)-(1.15) in
  \cite{chen_spectrum_1994} are satisfied).  Then there
  exist $0<\varepsilon_{0} < 1$ and positive constant
  $C_{0}$ such that the principle eigenvalue of the
  linearized Cahn-Hilliard operator
  $\mathcal{L}_{CH}:=\Delta(\varepsilon\Delta-\frac{1}{\varepsilon}f'(\phi)I)$
  satisfies for all $t\in [0,T]$
	\begin{equation}\label{eq:AP:eigen}
	\lambda_{CH}=\inf_{\substack{0\neq v\in H^{1}(\Omega)\\ \Delta\mu=v}}
	\frac{\varepsilon\|\nabla v\|^{2}+\frac{1}{\varepsilon}(f'(\phi(\cdot,t))v,v)}{\|\nabla\mu\|^{2}}
	\geq-C_{0},
	\end{equation}
	for all $\varepsilon\in (0,\varepsilon_{0})$.
\end{lemma}

\begin{theorem}\label{thm:errorHn1}
  Suppose all of the Assumption \ref{ap:1}(i)(ii),\ref{ap:2}, \ref{ap:3} hold. Let time
  step $\tau$ satisfy the following constraint
  \begin{equation}\label{cet18-00}
      \tau \le \min \left\{ \frac{1}{4(C_0+L^2)},
	      \frac{\varepsilon^6}{12(B^2\varepsilon^2+L^2)},
        C_5\varepsilon^{(5+\frac{1}{2}\max\{\rho_6 +4,\rho_4+6,\rho_2+9\}+\frac{d-2}{8})\frac{8}{18-d}} 
      \right\},
  \end{equation}
  where $C_5$ is a constant can be bounded uniformly independent of $\tau$ and $\varepsilon$.
  Then the solution of (\ref{eq:BDF:1}) (\ref{eq:BDF:2})
  satisfies the following error estimate
   \begin{equation}\label{eq:ES:fine0}
   \begin{split}
   &\max_{1\le n\le N}\left\{\|e^{n+1}\|_{-1}^{2} +\|2e^{n+1}-e^{n}\|_{-1}^{2}
   +2A\tau^{2}\|\nabla e^{n+1}\|^{2}\right\}\\
   &+\sum_{n=1}^{N}\left(2A\tau^2\|\delta_{t}\nabla e^{n+1}\|^{2}
   +\tau\frac{\varepsilon^4}{2}\|\nabla e^{n+1}\|^{2}\right)\\
   \lesssim{} &
   \varepsilon^{-\max\{\rho_6 +4,\rho_4+6,\rho_2+9\}}\tau^{4}.
   \end{split}
   \end{equation}   
\end{theorem}

\begin{proof}
  We refine the result of Theorem \ref{prio} by
  re-estimating $J_4$ in equation \eqref{eq:eoee} as
	\begin{equation}\label{eq:J4_2}
	\begin{split}
	J_{4}&=-B(\delta_{tt}e^{n+1},e^{n+1})
	 \leq 
	\frac{B^2}{\eta_1}\|\delta_{tt}e^{n+1}\|^2_{-1} 
	+\frac{\eta_1}{4}\|\nabla e^{n+1}\|^2,
	\end{split}
	\end{equation}
	and rewriting $J_5$ as
	\begin{align}\label{eq:J5_2}
	J_5 &=  J_6 + J_7,\\
	\begin{split}\label{eq:J6}
	J_{6}&=-\frac{1}{\varepsilon}\left(f(2\phi^{n}-\phi^{n-1})-f(\phi^{n+1}),e^{n+1}\right)\\
	&\le \frac{L}{\varepsilon}(|\delta_{tt}e^{n+1}|+|R_{2}^{n+1}|, |e^{n+1}|)\\
	&\le \frac{L^2}{\varepsilon^2\eta_1}\left(\|\delta_{tt}e^{n+1}\|_{-1}^2
	+\|R_{2}^{n+1}\|_{-1}^2\right) + \frac{\eta_1}{2} \|\nabla e^{n+1}\|^2,
	\end{split}\\
	\begin{split}\label{eq:J7}
	J_{7}&=
	-\frac{1}{\varepsilon}\left(f(\phi^{n+1})-f(\phi(t^{n+1})),e^{n+1}\right)\\
	&\le -\frac{1}{\varepsilon}\left(f'(\phi(t^{n+1}))e^{n+1},e^{n+1}\right)
	+\frac{L_2}{\varepsilon}\|e^{n+1}\|_{L^3}^3.
	\end{split}
	\end{align}
The spectrum estimate \eqref{eq:AP:eigen} give us
  \begin{equation}\label{spectrum}
    \varepsilon \|\nabla e^{n+1}\|^{2}
    +\frac{1}{\varepsilon}(f'(\phi(t^{n+1}))e^{n+1},  e^{n+1})
    \geq -C_{0}\|e^{n+1}\|_{-1}^{2}.
  \end{equation}
  If the entire $\varepsilon \|\nabla e^{n+1}\|^{2}$
  term is used to control the the term involving $f'$ in $J_7$, we
  will not be able to control the $\|\nabla e^{n+1}\|$ terms
  in $J_1,\ldots, J_4$ and $J_6$.  So we apply
  (\ref{spectrum}) with a scaling factor $(1-\eta)$ close to
  but smaller than $1$, to get
  \begin{equation}\label{spectrum1}
    -(1-{\eta})\frac{1}{\varepsilon}(f'(\phi(t^{n+1}))e^{n+1},  e^{n+1})
    \leq (1-{\eta})C_{0} \|e^{n+1}\|_{-1}^{2}
    +(1-{\eta})\varepsilon\|\nabla e^{n+1}\|^{2}.
  \end{equation}
  On the other hand,
  \begin{equation}\label{spectrum2}
    \begin{split}
      -\frac{\eta }{\varepsilon}(f'(\phi(t^{n+1}))e^{n+1},
      e^{n+1})
      \leq &
      \frac{L^{2}}{\varepsilon^{2}}\frac{\eta }{\eta_2}\|e^{n+1}\|_{-1}^{2}
      +\frac{\eta \eta_2 }{4}\|\nabla e^{n+1}\|^{2}.
    \end{split}
  \end{equation}
Combining (\ref{eq:J7}), (\ref{spectrum1}) and (\ref{spectrum2}) together, we have
  \begin{equation}\label{eq:J7_2}
    \begin{split}
      J_{7} \leq
      \left(C_{0}(1-\eta )+\frac{L^{2}}{\varepsilon^{2}}\frac{\eta }{\eta_2} \right) \|e^{n+1}\|_{-1}^{2}
       +\frac{L_2}{\varepsilon}\|e^{n+1}\|^{3}_{L^{3}}
      +\left((1-\eta )\varepsilon+\frac{\eta \eta_2 }{4}\right)\|\nabla
      e^{n+1}\|^{2}.
    \end{split}
  \end{equation}
  Substituting the estimate of \eqref{l1}-\eqref{eq:J3},
  \eqref{eq:J4_2}-\eqref{eq:J6} and
  \eqref{eq:J7_2} into (\ref{eq:eoee}), we get
 	\begin{equation}\label{eq:withL3}
 	\begin{split}
 	&\frac{1}{4\tau}(\|e^{n+1}\|_{-1}^{2}+\|2e^{n+1}-e^{n}\|_{-1}^{2})
 	+\frac{1}{2}A\tau\|\nabla e^{n+1}\|^{2}\\
 	&+\frac{1}{2}A\tau\|\delta_{t}\nabla e^{n+1}\|^{2}
 	+\varepsilon\|\nabla e^{n+1}\|^{2}
 	+\frac{1}{4\tau}\|\delta_{tt}e^{n+1}\|_{-1}^{2}
 	\\
 	\leq{}& 
 	\frac{1}{4\tau}(\|e^{n}\|_{-1}^{2}+\|2e^{n}-e^{n-1}\|_{-1}^{2})
 	+\frac{1}{2}A\tau\|\nabla e^{n}\|^{2}\\
 	&+\frac{1}{\eta_1}\|\Delta^{-1}R_{1}^{n+1}\|_{-1}^{2}
 	+\left(B^2+\frac{L^{2}}{\varepsilon^{2}}\right)\frac{1}{\eta_1}\|R_{2}^{n+1}\|_{-1}^{2}
 	+\frac{A^{2}}{\eta_1}\|\nabla
 	R_{3}^{n+1}\|^{2}\\
 	&+\left(B^2+\frac{L^{2}}{\varepsilon^{2}}\right)\frac{1}{\eta_1} \|\delta_{tt}e^{n+1}\|^2_{-1} 
 	+ \frac{3}{2}\eta_1\|\nabla e^{n+1}\|^{2}\\
 	&+\left(C_{0}(1-\eta )+\frac{L^{2}}{\varepsilon^{2}}\frac{\eta }{\eta_2} \right) \|e^{n+1}\|_{-1}^{2}
 	+\frac{L_2}{\varepsilon}\|e^{n+1}\|^{3}_{L^{3}}
 	+\left((1-\eta )\varepsilon+\frac{\eta \eta_2 }{4}\right)\|\nabla
 	e^{n+1}\|^{2}.
    \end{split}
 	\end{equation}

    We now estimate the $L^3$ term.  By interpolating $L^3$
    between $L^2$ and $H^1$ and using Poincare's inequality
    for error function, we get
\begin{equation*}
	\|e^{n+1}\|^3_{L^3} \le K \| \nabla e^{n+1} \|^{\frac{d}{2}} \| e^{n+1} \|^{\frac{6-d}{2}},
\end{equation*}
where $K$ is a constant independent of $\varepsilon$ and
$\tau$. We continue the estimate by using
$\|e^{n+1}\|^2 \le \|\nabla e^{n+1}\|\cdot \|e^{n+1}\|_{-1}$
to get
\begin{equation}\label{eq:L3}
\frac{L_2}{\varepsilon}\|e^{n+1}\|^3_{L^3} \le {L_2}K\varepsilon^{-1} \| \nabla e^{n+1} \|^{\frac{d}{2}+\frac{6-d}{4}} \| e^{n+1} \|_{-1}^{\frac{6-d}{4}} = G^{n+1} \|\nabla e^{n+1} \|^2,
\end{equation}
where $G^{n+1}=L_2K\varepsilon^{-1} \| \nabla e^{n+1} \|^{\frac{d-2}{4}} \| e^{n+1} \|_{-1}^{\frac{6-d}{4}}$.

Now plug equation \eqref{eq:L3} into \eqref{eq:withL3}, and
take $\eta_2=\varepsilon$, $\eta =\varepsilon^3$,
$\eta_1={\varepsilon^4}/{3}$ and
$\tau \le {\varepsilon^6}/12(B^2\varepsilon^2+L^2)$,
such that
  \begin{equation*}
  \frac{3\eta_1}{2}+(1-\eta )\varepsilon+\frac{\eta \eta_2 }{4}
  =\varepsilon-\frac{1}{4}\varepsilon^4,\quad
  \left(B^2+\frac{L^2}{\varepsilon^2}\right)\frac{1}{\eta_1}\le \frac{1}{4\tau},
  \quad \frac{L^2}{\varepsilon^2}\frac{\eta }{\eta_2} = L^2,
  \end{equation*}
  we get
 	\begin{equation}\label{eq:withL3subst}
 	\begin{split}
 	&\frac{1}{4\tau}(\|e^{n+1}\|_{-1}^{2}+\|2e^{n+1}-e^{n}\|_{-1}^{2})
 	+\frac{1}{2}A\tau\|\nabla e^{n+1}\|^{2}\\
 	&+\frac{1}{2}A\tau\|\delta_{t}\nabla e^{n+1}\|^{2}
 	+\frac{\varepsilon^4}{4}\|\nabla e^{n+1}\|^{2}
 	\\
 	\leq{}& 
 	\frac{1}{4\tau}(\|e^{n}\|_{-1}^{2}+\|2e^{n}-e^{n-1}\|_{-1}^{2})
 	+\frac{1}{2}A\tau\|\nabla e^{n}\|^{2}\\
 	&+\frac{3}{\varepsilon^4}\|\Delta^{-1}R_{1}^{n+1}\|_{-1}^{2}
 	+\left(B^2+\frac{L^{2}}{\varepsilon^{2}}\right)\frac{3}{\varepsilon^4}\|R_{2}^{n+1}\|_{-1}^{2}
 	+\frac{3A^{2}}{\varepsilon^4}\|\nabla
 	R_{3}^{n+1}\|^{2}
 	\\
 	&+\left(C_{0}(1-\eta )+{L^{2}} \right) \|e^{n+1}\|_{-1}^{2}
 	+G^{n+1}\| \nabla e^{n+1} \|^{2}.
 	\end{split}
 	\end{equation}
    If $G^{n+1}$ is uniformly bounded by constant
    $\varepsilon^4/8$, then we can get a finer error
    estimate for $\tau < {1}/{8(C_0+L^2)}$ by using discrete Gronwall inequality and the assumption of first step error
    \eqref{eq:AP:phi1e}:
 \begin{equation}\label{eq:ES:fine}
 \begin{split}
 &\max_{1\le n\le N}\left\{\|e^{n+1}\|_{-1}^{2} +\|2e^{n+1}-e^{n}\|_{-1}^{2}
 +2A\tau^{2}\|\nabla e^{n+1}\|^{2}\right\}\\
 &\qquad+\sum_{n=1}^{N}\left(2A\tau^2\|\delta_{t}\nabla e^{n+1}\|^{2}
 +\tau\varepsilon^4/2\|\nabla e^{n+1}\|^{2}\right)\\
 \leq{} &
 \frac{3}{2}C'_{3}\exp(8(C_0+L^2)T)
 \varepsilon^{-\max\{\rho_6 +4,\rho_4+6,\rho_2+9\}}\tau^{4}, \quad N\ge 1,
 \end{split}
 \end{equation}
 where $C'_3 = C_3+ C_1(5 + 2A\varepsilon^{-1}\tau)$.
 We prove this by induction. Assuming that the above
 estimate holds for all first $N-1$ time steps. Then the
 coarse estimate \eqref{eq:ES:coarse1} leads to
 	\begin{equation}\label{eq:ES:En1}
 	\begin{split}
 	\|e^{N+1}\|_{-1}^{2} + \tau \varepsilon\|\nabla e^{N+1} \|^2
 	\leq{} & 
 	\|e^{N}\|_{-1}^{2}
 	+(1+C_{2}\tau\varepsilon^{-3})
 	\|2e^{N}-e^{N-1}\|_{-1}^{2}
 	\\
 	&+2A\tau^{2}\|\nabla e^{N}\|^{2}
 	+C_{3}\tau^{4}\varepsilon^{-\max\{\rho_6 +1,\rho_4+3,\rho_2+6 \}}.\\ 	
 	\end{split}
 	\end{equation}
    Then by induction assumption \eqref{eq:ES:fine}, we get
\begin{equation}\label{eq:ES:En2}
\begin{split}
\|e^{N+1}\|_{-1}^{2} + \tau \varepsilon\|\nabla e^{N+1} \|^2
\leq{} & 
C_4\varepsilon^{-\max\{\rho_6 +4,\rho_4+6,\rho_2+9 \}}\tau^{4} 
\end{split}
\end{equation}
where
$C_4 = \frac{3}{2}C'_3(1+C_2\tau\varepsilon^{-3}) \exp(8(C_0+L^2)T) + C_3
\varepsilon^3$.  Thus, if
 	\begin{equation}\label{eq:GnE}
 	 L_2K \varepsilon^{-1}  [C_4\varepsilon^{-\max\{\rho_6 +4,\rho_4+6,\rho_2+9 \}-1} \tau^3 ]^{\frac{d-2}{8}} 
 	 [C_4\varepsilon^{-\max\{\rho_6 +4,\rho_4+6,\rho_2+9 \}}\tau^4]^{\frac{6-d}{8}}\le \varepsilon^4/8,
 	\end{equation}
 we get $G^{N+1} \le \varepsilon^4/8$. Solving \eqref{eq:GnE}, we get
 \begin{equation}\label{eq:tGnE}
 \tau
 \le  C_5 \varepsilon^{(5+\frac{1}{2}\max\{\rho_6 +4,\rho_4+6,\rho_2+9 \}+\frac{d-2}{8})\frac{8}{18-d}},
 \end{equation} 
 where $C_5=\left( 8L_2KC_4^\frac12\right)^{-\frac{8}{18-d}}$.	 	
 The proof is completed. \qed
\end{proof}

\begin{remark}\label{rmk:errest}
  Different to the related work \cite{feng_error_2004}, we
  did not take $\eta=\tau^\beta$ in equation
  \eqref{eq:withL3}, this allows us to maintain the full
  order about $\tau$. Besides, our induction method to
  handle $L^3$ term is much simpler than the method used in
  \cite{feng_error_2004} and \cite{kessler_posteriori_2004},
  where first-order schemes are studied.
\end{remark}

\begin{remark}\label{rmk:errest2}
  Theorem \ref{thm:errorHn1} and its proof is valid for the
  special cases $A=0$ and/or $B=0$,
  since the condition \eqref{eq:BDF:ABcond} is not used in
  the proof.
\end{remark}

\section{Implementation and numerical results}

In this section, we numerically verify our
schemes are second order accurate in time and energy
stable.

We use the commonly used  double-well potential 
$F(\phi)=\frac{1}{4}(\phi^{2}-1)^{2}$.
It is a common practice to modify $F(\phi)$ to have a
quadratic growth for $|\phi|>1$ (since physically
$|\phi|\leq 1$), such that a global Lipschitz condition is
satisfied \cite{shen_numerical_2010},
\cite{condette_spectral_2011}. To get a $C^4$ smooth
double-well potential with quadratic growth, we introduce
$\tilde{F}(\phi) \in C^{\infty}(\mathbf{R})$ as a smooth
mollification of
\begin{equation}\label{efnew}
\hat{F}(\phi)=
\begin{cases}
\frac{11}{2}(\phi-2)^{2}+6(\phi-2)+\frac94, &\phi>2, \\
\frac{1}{4 }(\phi^{2}-1)^{2}, &\phi\in [-2,2], \\
\frac{11}{2}(\phi+2)^{2}+6(\phi+2)+\frac94, &\phi<-2.
\end{cases}
\end{equation}
with a mollification parameter much smaller than 1, to
replace $F(\phi)$. Note that the truncation points $-2$ and
$2$ used here are for convenience only. Other values outside
of region $[-1,1]$ can be used as well.  For simplicity, we
still denote the modified function $\tilde{F}$ by $F$.

\subsection{Space discrete and implementation}

To test the numerical scheme, we solve \eqref{eq:CH} in
a 2-dimensional tensor product domain
$\Omega=[-1,1]\times[-1,1]$. We use a Legendre Galerkin
method similar as in
\cite{shen_efficient_2015,yu_numerical_2017} for spatial
discretization.  Let $L_k(x)$ denote the Legendre polynomial
of degree $k$. We define
\[
V_M = \mbox{span}\{\,\varphi_k(x)\varphi_j(y),\ k,j=0,\ldots, M-1\,\}\in H^1(\Omega),
\]
where
$\varphi_0(x) = L_0(x); \varphi_1(x)=L_1(x);  \varphi_k(x)=L_k(x)-L_{k+2}(x), k=2,\ldots, M\!-\!1
$,
 be the Galerkin approximation space for both $\phi^{n+1}$ and $\mu^{n+1}$. Then the full discretized
form for the SL-BDF2 scheme reads:

Find $(\phi^{n+1}, \mu^{n+1}) \in (V_M)^{2}$ such that
\begin{equation}\label{eq:bdf2:fulldis1}
\frac{1}{2\tau}(3\phi^{n+1}-4\phi^{n}+\phi^{n-1}, \omega)=-\gamma(\nabla \mu^{n+1},\nabla \omega), 
\quad \forall\, \omega \in V_M,
\end{equation}
\begin{equation}\label{eq:bdf2:fulldis2}
\begin{split}
(\mu^{n+1},\varphi)={} &\varepsilon (\nabla \phi^{n+1},\nabla \varphi) +
\frac{1}{\varepsilon}(f(2\phi^{n}-\phi^{n-1}),\varphi)
+A\tau (\nabla \delta_{t}\phi^{n+1},\nabla \varphi)\\
&\qquad\qquad\qquad\qquad\qquad
+B(\delta_{tt}\phi^{n+1},\varphi), \quad \forall\, \varphi \in
V_M.
\end{split}
\end{equation}
This is a linear system with constant coefficients for
$(\phi^{n+1}, \mu^{n+1})$, which can be efficiently solved.
We use a spectral transform with double quadrature points to
eliminate the aliasing error and efficiently evaluate the
integration $(f(2\phi^n-\phi^{n-1}),\varphi)$ in equation
\eqref{eq:bdf2:fulldis2}.

We take $\varepsilon=0.05$ and $M=127$ and use two different
initial values to test the stability and accuracy of the
proposed schemes:
\begin{enumerate}
\item[(i)] $\{\phi_0(x_i,y_j) \}\in\, {\bf{R}}^{2M\times2M}$
  with $x_i, y_j$ are tensor product Legendre-Gauss
  quadrature points and $\phi_0(x_i,y_j)$ is a uniformly
  distributed random number between $-1$ and $1$ (shown in
  the left picture of Fig. \ref{fig1});
\item[(ii)] The solution of the Cahn-Hilliard equation at
  $t=64\varepsilon^3$ which takes $\phi_0$ as its initial
  value (Denoted by $\phi_1$ shown in the middle picture of
  Fig. \ref{fig1}).
\end{enumerate}
\begin{figure}
	\centering
	\includegraphics[width=\textwidth]{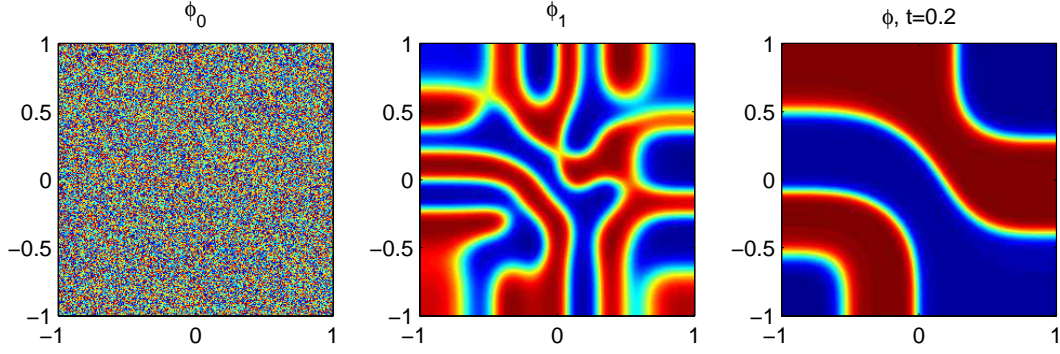}
	\caption{The two random initial values $\phi_0$,
      $\phi_1$ and the state of $\phi_1$ evolves $0.2$ time
      unit according to the Cahn-Hilliard equation
      \eqref{eq:CH} with $\gamma=1$.}
	\label{fig1}
\end{figure}

Given $\phi_0$, to start the second order schemes, we use following first order stabilized scheme to generate $\phi^1$
\begin{equation}
\frac{\varphi^{n+1}-\varphi^n}{s}=\Delta \omega^{n+1}, \quad 
\omega^{n+1}=-\varepsilon\Delta\varphi^{n+1}+\frac{1}{\varepsilon}f(\varphi^n)+S\delta_t\varphi^{n+1},\quad n=0,\ldots, m-1,
\end{equation}
where $S = 1/\varepsilon$ is a stabilization constant,
$s$ is the time step-size, $\varphi^0=\phi_0$. To get an accurate $\phi^1$,
we let $s=\tau/m$ with $m =10$, and let $\phi^1=\varphi^m$. 

\subsection{Stability results}

Table \ref{tstab0:B},\ref{tstab0:A} show the required
minimum values of $A$ (resp. $B$) with different $\gamma$,
$B$ (resp. $A$) and $\tau$ values for stably solving (the increase of discrete energy in each time-step is less than $10^{-10}$ for 1024 time steps) the Cahn-Hilliard equation
\eqref{eq:CH} with initial value $\phi_0$.  The results for
the initial value $\phi_1$ are similar. From the two tables,
we observe that: 
\begin{enumerate}
	\item For smaller $\tau$ values, the SL-BDF2
	scheme need smaller $A,B$ to be stable comparing to the SL-CN scheme, while both of
	them are stable with $A=0, B=0$ when $\tau$ is small enough.
	On the other hand, for larger $\tau$, SL-BDF2 schemes need
	relatively larger $A$ and $B$ than SL-CN scheme. This might
	due to the fact that the SL-BDF2 scheme has larger diffusion
	and splitting error than SL-CN scheme.
	\item The existence of a non-zero $B$ 
	 remarkably reduce the values of $A$ needed for the scheme to be energy stable, especially when $\gamma=1$ and $10^{-5}\le\tau\le 0.1$. On the other hand, a non-zero $A$ 
	remarkably reduce the values of $B$ needed for the scheme to be energy stable when $\tau$ is large.
\end{enumerate}

\begin{table}
	\begin{tabular}{|c|c|c|c|c|c|c|c|c|}
		\hline
		\multirow{3}{*}{$\tau$ }
		&  \multicolumn{4}{|c|}{SL-BDF2}
		&  \multicolumn{4}{|c|}{SL-CN}
		\\ \cline{2-9}&  \multicolumn{2}{|c|}{$\gamma=0.0025$ }
		&  \multicolumn{2}{|c|}{$\gamma=1$ }
		&  \multicolumn{2}{|c|}{$\gamma=0.0025$ }
		&  \multicolumn{2}{|c|}{$\gamma=1$ }
		\\ \cline{2-9} & $B=0$ & $B=10$ & $B=0$ & $B=10$ & $B=0$ & $B=10$ & $B=0$ & $B=10$\\ \hline
		10  & 1 & 0.5 & 12.5 & 12.5 & 1 & 0.25 & 12.5 & 12.5\\ \hline
		1  & 2 & 0.5 & 25 & 25 & 1 & 0.5 & 25 & 12.5\\ \hline
		0.1  & 1 & 0.25 & 200 & 100 & 1 & 0.25 & 200 & 50\\ \hline
		0.01  & 0 & 0 & 400 & 200 & 1 & 0.25 & 400 & 100\\ \hline
		0.001  & 0 & 0 & 800 & 200 & 0 & 0 & 400 & 100\\ \hline
		0.0001  & 0 & 0 & 200 & 0 & 0 & 0 & 400 & 100\\ \hline
		1E-05  & 0 & 0 & 0 & 0 & 0 & 0 & 400 & 50\\ \hline
		1E-06  & 0 & 0 & 0 & 0 & 0 & 0 & 0 & 0\\ \hline
	\end{tabular} 
	\centering 
	\caption{The minimum values of $A$ (only values $\{0, 2^i,i=-7,\ldots,1\}\times4\gamma/\varepsilon^2$ are tested) 
		to make SL-BDF2 and SL-CN scheme stable when $\gamma$, $B$ and $\tau$ taking different values.}
	\label{tstab0:B}
\end{table}

\begin{table}
	\begin{tabular}{|c|c|c|c|c|c|c|c|c|}
		\hline
		\multirow{3}{*}{$\tau$ }
		&  \multicolumn{4}{|c|}{SL-BDF2}
		&  \multicolumn{4}{|c|}{SL-CN}
		\\ \cline{2-9}&  \multicolumn{2}{|c|}{$\gamma=0.0025$ }
		&  \multicolumn{2}{|c|}{$\gamma=1$ }
		&  \multicolumn{2}{|c|}{$\gamma=0.0025$ }
		&  \multicolumn{2}{|c|}{$\gamma=1$ }
		\\ \cline{2-9} & $A=0$ & $A=0.0625$ & $A=0$ & $A=25$ & $A=0$ & $A=0.0625$ & $A=0$ & $A=25$\\ \hline
		10  & 320 & 40 & $>640$ & 0 & 320 & 20 & $>640$ & 0\\ \hline
		1  & 40 & 40 & $>640$ & 0 & 80 & 20 & $>640$ & 0\\ \hline
		0.1  & 20 & 20 & $>640$ & 40 & 20 & 20 & 640 & 20\\ \hline
		0.01  & 0 & 0 & 40 & 40 & 20 & 20 & 320 & 20\\ \hline
		0.001  & 0 & 0 & 40 & 40 & 0 & 0 & 40 & 20\\ \hline
		0.0001  & 0 & 0 & 10 & 10 & 0 & 0 & 20 & 20\\ \hline
		1E-05  & 0 & 0 & 0 & 0 & 0 & 0 & 20 & 20\\ \hline
		1E-06  & 0 & 0 & 0 & 0 & 0 & 0 & 0 & 0\\ \hline
	\end{tabular} 
	\centering 
	\caption{The minimum values of $B$ (only values $\{0, 2^i,i=-3,\ldots,4\}\times 2/\varepsilon$ are tested) to make scheme SL-BDF2 and SL-CN stable when $\gamma$, $A$ and $\tau$ taking different values.}
	\label{tstab0:A}
\end{table}

Figure \ref{fig:2Energy} presents the discrete energy
dissipation of the SL-CN and SL-BDF2 scheme using several
time step-sizes. We see the energy decaying property is
maintained.

\begin{figure}
	\centering 
	\includegraphics[width=0.48\textwidth]{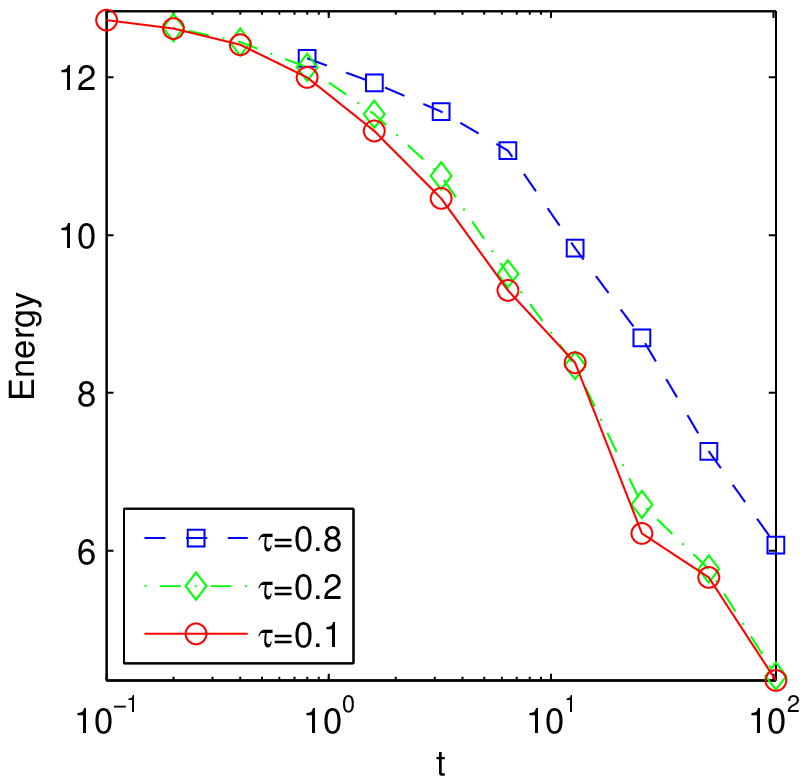}
	\includegraphics[width=0.48\textwidth]{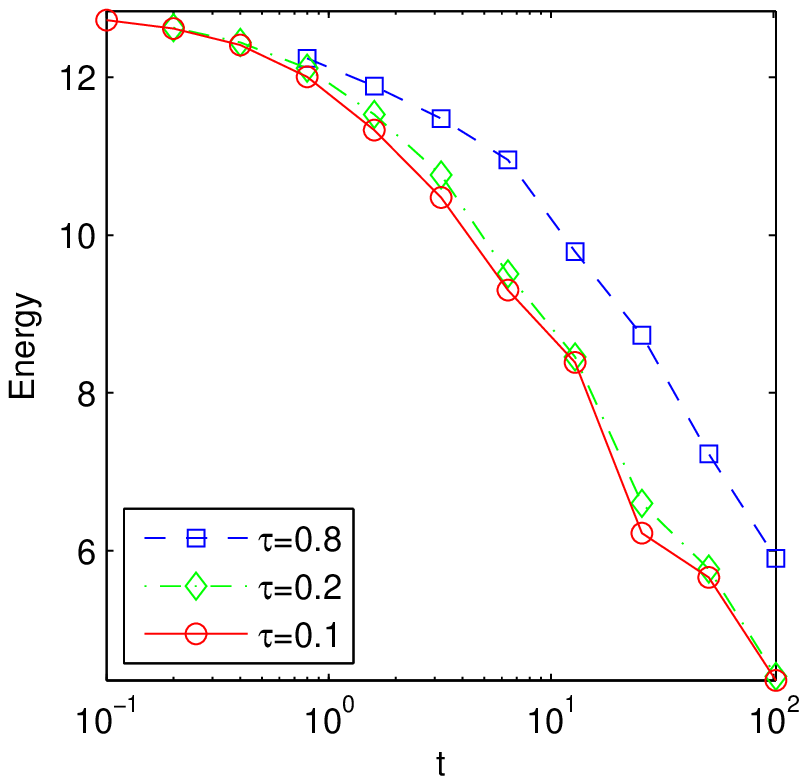}
	\caption{The discrete energy dissipation of the two
      schemes solving the Cahn-Hilliard equation with
      initial value $\phi_1$, and relaxation parameter
      $\gamma=0.0025$. Stability constant $A=0.25, B=20$ are
      used.  Left) result of SL-CN scheme; Right) result of
      SL-BDF2 scheme.}
	\label{fig:2Energy}
\end{figure}

\subsection{Accuracy results}

We take initial value $\phi_1$ to test the accuracy of the
two schemes.  The Cahn-Hilliard equation with
$\gamma=0.0025$ are solved from $t=0$ to $T=12.8$. To
calculate the numerical error, we use the numerical result
generated using $\tau=10^{-3}$ as a reference of exact
solution.  The results are given in Table
\ref{tbl:conv:SL-BDF2} and Table \ref{tbl:conv:SL-CN}. We
see that the schemes are both second order accuracy in
$H^{-1}, L^2$ and $H^1$ norm.  
\begin{table}
	\begin{tabular}{|c|c|c|c|c|c|c|}
		\hline
		$\tau$ & $H^{-1}$ Error & Order & $L^2$ Error & Order & $H^1$ Error & Order\\ \hline
		0.16 & 3.32E-02 &  & 2.63E-01 &  & 3.49E+00 &  \\ \hline
		0.08 & 9.71E-03 & 1.77 & 8.02E-02 & 1.72 & 1.10E+00 & 1.67 \\ \hline
		0.04 & 2.54E-03 & 1.94 & 2.10E-02 & 1.93 & 2.89E-01 & 1.93 \\ \hline
		0.02 & 6.38E-04 & 1.99 & 5.25E-03 & 2.00 & 7.21E-02 & 2.00 \\ \hline
		0.01 & 1.58E-04 & 2.02 & 1.30E-03 & 2.02 & 1.78E-02 & 2.02 \\ \hline
		0.005 & 3.80E-05 & 2.05 & 3.14E-04 & 2.05 & 4.30E-03 & 2.05 \\ \hline
	\end{tabular} 
	\centering 
	\caption{The convergence of the SL-BDF2 scheme
		with $B=40$, $A=0.25$
		for the Cahn-Hilliard equation with initial value $\phi_1$, parameter $\gamma=0.0025$.
		The errors are calculated at $T=12.8$ for both schemes.}
	\label{tbl:conv:SL-BDF2}
\end{table}

\begin{table}
	\begin{tabular}{|c|c|c|c|c|c|c|}
		\hline
		$\tau$ & $H^{-1}$ Error & Order & $L^2$ Error & Order & $H^1$ Error & Order\\ \hline
		0.16 & 3.26E-02 &  & 2.58E-01 &  & 3.42E+00 &  \\ \hline
		0.08 & 9.32E-03 & 1.81 & 7.64E-02 & 1.76 & 1.04E+00 & 1.71 \\ \hline
		0.04 & 2.41E-03 & 1.95 & 1.98E-02 & 1.95 & 2.71E-01 & 1.95 \\ \hline
		0.02 & 6.06E-04 & 1.99 & 4.96E-03 & 2.00 & 6.77E-02 & 2.00 \\ \hline
		0.01 & 1.50E-04 & 2.01 & 1.23E-03 & 2.01 & 1.68E-02 & 2.01 \\ \hline
		0.005 & 3.64E-05 & 2.05 & 2.97E-04 & 2.05 & 4.06E-03 & 2.05 \\ \hline
	\end{tabular} 
	\centering 
	\caption{The convergence of the SL-CN scheme
		with $B=20$, $A=0.25$
		for the Cahn-Hilliard equation with initial value $\phi_1$, parameter $\gamma=0.0025$.
		The errors are calculated at $T=12.8$ for both schemes.}
	\label{tbl:conv:SL-CN}
\end{table}

\section{Conclusions}

We propose two second order stabilized linear schemes
(SL-BDF2 scheme and SL-CN scheme) for the phase-field
Cahn-Hilliard equation. In both schemes, the nonlinear bulk
forces are treated explicitly with two additional linear
stabilization terms: $-A\tau \Delta\delta_t\phi^{n+1}$ and $B\delta_{tt}\phi^{n+1}$. In particular, the introduction of a
$H^1$ stabilization term $A\tau \Delta\delta_t\phi^{n+1}$
enables us to prove the unconditionally stability results. We
also give a rigorous optimal error analysis of the SL-BDF2
scheme. This error analysis holds for the special case $A=0$
and/or $B=0$ as well.  Numerical results are presented to
verify the stability and accuracy of the proposed schemes.
By combining nonzero values of $B$ and $A$, the two schemes
can obtain better stability results than the cases use only one stabilization term.

\section*{Acknowledgment}
The authors would like to thank Prof. Jie Shen and Prof. Xiaobing Feng for helpful discussions.
This work is partially supported by NNSFC
under Grant 11771439, 11371358, 91530322.

\appendix
\renewcommand{\theequation}{A.\arabic{equation}}

\section*{Appendix: Proof of Lemma \ref{lm:reg}}
\begin{proof}
	We first write down some inequalities that will be frequently used. The first one is the Holder's inequality
	\begin{equation}\label{eq:Holder3}
	\| u v w\|_{L^s} \le \|u\|_{L^p} \|v\|_{L^q} \| w\|_{L^r}, \quad\forall\ p,q,r\in(0,\infty],\ \frac{1}{s} = \frac{1}{p} + \frac{1}{q} + \frac{1}{r}.
	\end{equation}The second one is the Sobolev inequality
	\begin{equation}\label{eq:Sobolev}
	\| u \|_{L^q} \le C_s \| u \|_1, 
	\end{equation}
	where $q \in [2, \infty)$ for $d=2$;  $q\in [2, \frac{2d}{d-2}]$ for $d>2$; $C_s$ is a general constant independent of $\phi$. 
	We can further use Poincare's inequality to get
	\begin{equation} \label{eq:Sobolev2}
	\| v \|_{L^q} \le C_s \| \nabla v \|,\quad \forall v\in L^2_0(\Omega).
	\end{equation}
	For $v \in L^2_0(\Omega)$,  we also have following inequality
	\begin{equation}\label{eq:L02}
	\| v\|^2 = (\nabla v,\nabla (-\Delta)^{-1} v) \le \frac{1}{2\delta} \|\nabla v\|^2 + \frac{\delta}{2} \| v \|_{-1}^2,
	\end{equation} where $\delta>0$ is an arbitrary constant.
	
	Now, we begin the proof.
	\begin{enumerate}
		\item[(i)]
		When $\gamma=1$, we have Cahn-Hilliard equation
		\begin{equation}\label{eq:CH0}
		\phi_{t}+ \varepsilon \Delta^2 \phi =\dfrac{1}{\varepsilon} \Delta f(\phi).
		\end{equation}
		Multiplying \eqref{eq:CH0} by $-\Delta^{-1} \phi_t$ and using integration by parts, we get
		\begin{equation}\label{eq:CH1}
		\|\phi_t\|_{-1}^2 + \frac{\varepsilon}{2}\frac{d}{ d t} \|\nabla \phi\|^2
		=-\frac{1}{\varepsilon}(f(\phi),\phi_t)
		=-\frac{1}{\varepsilon}\frac{d}{d t}\int_{\Omega} F(\phi) dx.
		\end{equation}
		After integrating over $[0, T]$,  we obtain
		\begin{equation}
		\int_{0}^{T}\|\phi_t\|_{-1}^{2}{\rm d}t + 
		E_{\varepsilon}(\phi(T)) = E_{\varepsilon}(\phi^0)
		\end{equation} 
		Taking maximum values of terms on the left hand side for $T \in [0, \infty]$, we get the first part of (i) from \eqref{eq:AP:E0}.
		From the definition of $E_\varepsilon(\phi)$, and assumption \eqref{eq:AP:Fcoercive} we know
		\begin{equation}
		\|\phi\|_{L^2}^2 \le B_0 |\Omega| + B_1 \varepsilon^{-\sigma_1+1} \lesssim \varepsilon^{-(\sigma_1-1)^+}.
		\end{equation}
		Combining above estimate with the fact $\frac\varepsilon2\| \nabla \phi \| ^2 \lesssim \varepsilon^{-\sigma_1}$, we get 
		\begin{equation} \label{eq:CHregH1}
		\| \phi \|_1^2 \lesssim \varepsilon^{-(\sigma_1+1)}.
		\end{equation}
		\item[(ii)]
		We formally differentiate \eqref{eq:CH0} in time to obtain
		\begin{equation}\label{eq:CH00}
		\phi_{tt}+ \varepsilon \Delta^2 \phi_{t} = \dfrac{1}{\varepsilon}\Delta \left( f'(\phi)\phi_t \right). \\
		\end{equation}
		
		Pairing \eqref{eq:CH00} with $ -\Delta^{-1}\phi_t$ and using \eqref{eq:L02}, yields
		\begin{equation}\label{eq:CH13}
		\begin{split}
		\frac{1}{2}\frac{d}{d t}\|\phi_{t}\|_{-1}^2
		+ \varepsilon \| \nabla \phi_{t}\|^2
		= &-\dfrac{1}{\varepsilon}\left( f'(\phi)\phi_t, \phi_t \right)
		\leq \frac{\tilde{c}_0}{\varepsilon}\|\phi_t\|^2\\
		\leq & \frac{\varepsilon}{2} \|\nabla \phi_{t}\|^2 + \frac{\tilde{c}_{0}^2}{2 \varepsilon^3}\|\phi_{t}\|_{-1}^2.\\
		\end{split}
		\end{equation}
		Integrating \eqref{eq:CH13} over $[0,T]$ and taking maximum values for terms depending on $T$, we get
		\begin{equation}\label{eq:CH14}
		\esssup \limits_{t\in[0,\infty]}  \|\phi_{t}\|_{-1}^2+ \varepsilon\int_{0}^{\infty} \|\nabla \phi_{t}\|^2{\rm d} t
		\lesssim \frac{\tilde{c}_{0}^2}{ \varepsilon^3} \int_{0}^{\infty } \|\phi_t\|_{-1}^2 {\rm d} t
		+\|\phi_{t}^0\|_{-1}^2. \\
		\end{equation}
		The assertion then follows from (i) and the inequality \eqref{eq:AP:htHn1} of Assumption \ref{ap:2}.
		
      \item[(iii)]
		Testing \eqref{eq:CH00} with $\phi_t$, using \eqref{eq:Holder3} and \eqref{eq:Sobolev} with Poincare's inequality, we get
		\begin{equation}\label{eq:CH21}
		\begin{split}
		\frac{1}{2} \frac{d}{d t}\| \phi_t\|^2 + \varepsilon \|\Delta \phi_t \|^2
		= & \frac{1}{\varepsilon}(f'(\phi) \phi_t, \Delta \phi_t)
		\leq \frac{1}{\varepsilon} \|f'(\phi)\|_{L^3} \| \phi_t\|_{L^6} \| \Delta \phi_t\|\\
		\leq &
		\frac{\varepsilon}{2} \|\Delta \phi_t\|^2
		+\frac{1}{2 \varepsilon^3} \|f'(\phi)\|_{L^3}^2 \| \phi_t \|_{L^6}^2\\
		\leq &
		\frac{\varepsilon}{2} \|\Delta \phi_t\|^2
		+\frac{C_s^2}{2 \varepsilon^3} \|f'(\phi)\|_{L^3}^2 \| \nabla \phi_t \|^2,
		\end{split}
		\end{equation}
		which leads to 
		\begin{equation}\label{eq:CH22}
		\esssup \limits_{t\in[0,\infty]}\|\phi_t\|^2 + \varepsilon \int_{0}^{\infty} \|\Delta \phi_t\|^2 {\rm d} t
		\lesssim
		\frac{C_s}{\varepsilon^3}  \esssup_{t\in[0,\infty]} \|f'(\phi)\|_{L^3}^2  
		\int_{0}^{\infty} \|\nabla \phi_t\|^2 {\rm d} t
		+\|\phi_t^0\|^2.
		\end{equation}
		On the other hand side, by assumption \eqref{eq:AP:fp}, the Sobolev inequality \eqref{eq:Sobolev} and estimate \eqref{eq:CHregH1}, we have
		\begin{equation}\label{eq:fpL3}
		\|f'(\phi)\|_{L^3}^2 \lesssim \tilde{c}_2 \|\phi\|_{L^{3(p-2)}}^{2(p-2)} +\tilde{c}_3  
		\lesssim \tilde{c}_2 \| \phi \|_1^{2(p-2)} + \tilde{c}_3
		\lesssim \varepsilon^{-(\sigma_1+1)(p-2)}
		\end{equation}
		The assertion then follows from \eqref{eq:CH22}, \eqref{eq:fpL3},  (ii) and assumption \eqref{eq:AP:htL2}.
		
		\item[(iv)]
		Testing \eqref{eq:CH00} with $-\Delta^{-1}\phi_{tt}$, we get
		\begin{equation}\label{eq:CH23}
		\begin{split}
		&\|\phi_{tt}\|_{-1}^2+ \frac{\varepsilon}{2}\frac{d}{dt} \|\nabla \phi_{t}\|^2
		=- \dfrac{1}{\varepsilon}(f'(\phi)\phi_t, \phi_{tt})\\
		= & -\frac{1}{2\varepsilon}\frac{d}{dt}(f'(\phi)\phi_t,\phi_t) 
		+ \frac{1}{2\varepsilon} (f''(\phi)\phi_t^2,\phi_t) \\
		\leq & -\frac{1}{2\varepsilon}\frac{d}{dt}(f'(\phi)\phi_t,\phi_t) 
		+ \frac{1}{2\varepsilon} \|f''\|_{L^{6}} \| \phi_t^2\|_{L^3} \|\phi_t\| \\
		\le & -\frac{1}{2\varepsilon}\frac{d}{dt}(f'(\phi)\phi_t,\phi_t) 
		+ \frac{C_s^2}{2\varepsilon} \|f''\|_{L^{6}} \| \nabla \phi_t\|^2 \| \phi_t\| \\
		\end{split}
		\end{equation}
		Integrate  \eqref{eq:CH23} over $[0,T]$, we continue the estimate as
		\begin{equation}\label{eq:CH24}
		\begin{split}
		&2\int_{0}^{T} \|\phi_{tt}\|_{-1}^2 {\rm  d} t
		+ \varepsilon\|\nabla \phi_{t}(T)\|^2
		- \varepsilon\|\nabla \phi_{t}^0\|^2\\
		\leq & -\frac{1}{\varepsilon}(f'(\phi)\phi_t,\phi_t)|_{t=T} + \frac{1}{\varepsilon} (f'(\phi^0)\phi_t^0,\phi_t^0)
		+ \frac{C_s^2}{\varepsilon} \esssup_{t\in[0,T]}\{\|f''\|_{L^{6}} \|\phi_t\|\}\int_0^T \|\nabla\phi_t\|^2 {\rm d}t\\
		\leq &  \frac{\varepsilon}{2} \|\nabla \phi_{t}(T)\|^2 + \frac{\tilde{c}_{0}^2}{2 \varepsilon^3}\|\phi_{t}(T)\|_{-1}^2
		+ \frac{1}{\varepsilon} (f'(\phi^0)\phi_t^0,\phi_t^0)
		+ \frac{C_s^2}{\varepsilon} \esssup_{t\in[0,T]}\{\|f''\|_{L^{6}} \|\phi_t\|\}\int_0^T \|\nabla\phi_t\|^2{\rm d}t, 
		\end{split}
		\end{equation}
		i.e. 
		\begin{equation}\label{eq:CH24-c}
		\begin{split}
		2\int_{0}^{T} \|\phi_{tt}\|_{-1}^2 {\rm  d} t
		&+ \frac{\varepsilon}{2}\|\nabla \phi_{t}(T)\|^2
		\leq 
		\varepsilon\|\nabla \phi_{t}^0\|^2 
		+ \frac{1}{\varepsilon} (f'(\phi^0)\phi_t^0,\phi_t^0)\\
		& + \frac{\tilde{c}_{0}^2}{2 \varepsilon^3}\|\phi_{t}(T)\|_{-1}^2
		+ \frac{C_s^2}{\varepsilon} \esssup_{t\in[0,T]} \{\|f''\|_{L^{6}} \|\phi_t\| \}\int_0^T \|\nabla\phi_t\|^2{\rm d}t.
		\end{split}
		\end{equation}
		On the other hand, by \eqref{eq:AP:fpp}, the Sobolev inequality \eqref{eq:Sobolev} and estimate \eqref{eq:CHregH1}, we have
		\begin{equation}\label{eq:fppL6}
		\|f''\|_{L^{6}} \lesssim \tilde{c}_4\| \phi \|_{L^{6(p-3)^+}}^{(p-3)^+} + \tilde{c}_5 
		\lesssim \|\phi\|_1^{(p-3)^+} 
		\lesssim \varepsilon^{-\frac12(\sigma_1+1)(p-3)^+}
		\end{equation}
		By taking maximum for terms depending on $T$ in \eqref{eq:CH24-c} and using \eqref{eq:fppL6}, (ii), (iii) and the inequality \eqref{eq:AP:LEt0} of Assumption \ref{ap:2}. we obtain the assertion (iv).

		\item[(v)]
		We formally differentiate \eqref{eq:CH00} in time to derive
		\begin{equation}\label{eq:CH000}
		\phi_{ttt}+ \varepsilon \Delta^2 \phi_{tt}
		=\dfrac{1}{\varepsilon}\Delta \left( f''(\phi)(\phi_t)^2+f'(\phi)\phi_{tt} \right). \\
		\end{equation}
		Testing \eqref{eq:CH000} with $\Delta^{-2} \phi_{tt}$, we obtian
		\begin{equation}\label{eq:CH31}
		\begin{split}
		& \frac{1}{2} \frac{d}{ dt } \| \Delta^{-1}\phi_{tt}\|^2 + \varepsilon \| \phi_{tt} \|^2
		=\dfrac{1}{\varepsilon} \left( f''(\phi)(\phi_t)^2+f'(\phi)\phi_{tt},  \Delta^{-1}\phi_{tt}\right)\\
		\leq{}&\dfrac{\varepsilon}{2} \|f''(\phi)\|^2_{L^2} \|\phi_t\|_{L^6}^{4} +
		\dfrac{1}{2\varepsilon^3} \|\Delta^{-1} \phi_{tt}\|^2_{L^6}
		+\frac{1}{2\varepsilon^3}\|f'(\phi)\|_{L^{3}}^2\|\Delta^{-1} \phi_{tt}\|_{L^6}^2
		+\frac{\varepsilon}{2} \|\phi_{tt}\|^2 \\
		\le{}& \dfrac{\varepsilon}{2} C_s^4\|f''(\phi)\|^2_{L^2} \|\nabla \phi_t\|^{4}
		+ \dfrac{C_s^2}{2\varepsilon^3}\| \phi_{tt}\|^2_{-1}
		+\frac{C_s^2}{2\varepsilon^3} \|f'(\phi)\|_{L^{3}}^2
		\| \phi_{tt}\|_{-1}^2
		+\frac{\varepsilon}{2} \|\phi_{tt}\|^2. \\
		\end{split}
		\end{equation}
		After taking integration from $[0,T]$ and taking maximum for terms
		depending on $T$, we have
		\begin{equation}\label{eq:CH32}
		\begin{split}
		& \esssup \limits_{t\in[0,\infty]}  \| \Delta^{-1}\phi_{tt}\|^2 + \varepsilon\int_{0}^{\infty} \| \phi_{tt} \|^2 {\rm d} t\\
		\lesssim{} & {\varepsilon}  \esssup \limits_{t\in[0,\infty]} \left(\|f''(\phi)\|^2_{L^2} \|\nabla \phi_t\|^{2}\right)
		\int_{0}^{\infty} \|\nabla \phi_t\|^{2}{\rm d} t \\
		&+ \frac{1}{\varepsilon^3}
		\left(\esssup_{t\in[0,\infty]} \|f'(\phi)\|_{L^{3}}^2+1\right)
		\int_{0}^{\infty} \| \phi_{tt}\|^2_{-1}{\rm d} t
		+ \| \Delta^{-1}\phi_{tt}^0\|^2. \\
		\end{split}
		\end{equation}
		The assertion then follows from \eqref{eq:fpL3}, the following estimate
		\begin{equation}\label{eq:fppL2}
		\|f''\|_{L^{2}}^2 \lesssim \tilde{c}_4\| \phi \|_{L^{2(p-3)^+}}^{(p-3)^+} + \tilde{c}_5 
		\lesssim \|\phi\|_1^{2(p-3)^+} 
		\lesssim \varepsilon^{-(\sigma_1+1)(p-3)^+},
		\end{equation}
		(ii), (iv) and the inequality \eqref{eq:AP:httHn2} of Assumption \ref{ap:2}.

		\item[(vi)]
		Pairing \eqref{eq:CH000} with $-\Delta^{-3} \phi_{ttt}$, we obtain
		\begin{equation}\label{eq:CH35}
		\begin{split}
		&\|\Delta^{-1}\phi_{ttt}\|_{-1}^2
		+ \frac{\varepsilon}{2}\frac{d}{d t}\| \phi_{tt}\|_{-1}^2\\
		= &-\dfrac{1}{\varepsilon}\left( f''(\phi)(\phi_t)^2+f'(\phi)\phi_{tt}, \Delta^{-2} \phi_{ttt} \right)\\
		\leq&  \dfrac{C_s^2}{\varepsilon^2} \left( \|f''(\phi)\|_{L^2}^2 \|\phi_{t}\|_{L^6}^4
		+ \|f'(\phi)\|_{L^3}^2 \|\phi_{tt}\|^2 \right)
		+\frac{1}{2C_s^2}\|\Delta^{-2}\phi_{ttt}\|_{L^6}^2\\
		\leq&  \dfrac{C_s^2}{\varepsilon^2} \left( C_s^4 \|f''(\phi)\|_{L^2}^2 \|\nabla \phi_{t}\|^4
		+\|f'(\phi)\|_{L^3}^2  \|\phi_{tt}\|^2 \right)
		+\frac{1}{2}\|\Delta^{-1}\phi_{ttt}\|_{-1}^2.\\
		\end{split}
		\end{equation}
		Integrating \eqref{eq:CH35} from  $[0,\infty)$, we have
		\begin{equation}\label{eq:CH36}
		\begin{split}
		& \int_{0}^{\infty} \|\Delta^{-1}\phi_{ttt}\|_{-1}^2 {\rm d} t
		+ \esssup \limits_{t\in[0,\infty]} \varepsilon \| \phi_{tt}\|_{-1}^2\\
		\leq&  \dfrac{2}{\varepsilon^2} C_s^6 \esssup \limits_{t\in[0,\infty]}
		\left(\|f''(\phi)\|^2_{L^2}\|\nabla \phi_t\|^{2} \right)
		\int_{0}^{\infty} \|\nabla \phi_t\|^{2} {\rm d} t\\
		& + \dfrac{2C_s^2}{\varepsilon^2}  \esssup \limits_{t\in[0,\infty]} \|f'(\phi)\|_{L^3}^2 \int_{0}^{\infty} \|\phi_{tt}\|^2  {\rm d} t
		+\varepsilon \| \phi_{tt}^0\|_{-1}^2.\\
		\end{split}
		\end{equation}
		The assertion then follows from \eqref{eq:fppL2}, \eqref{eq:fpL3}, (ii), (iv), (v) and the inequality \eqref{eq:AP:httHn1} of Assumption \ref{ap:2}.
		
	\end{enumerate} \qed
	
\end{proof}

\bibliographystyle{plain}  
\bibliography{LSS2}

\end{document}